\documentclass[final,reqno,twoside,a4paper,11pt]{amsart}
\usepackage{mltools}
\withdetails
\usepackage{math62,mlthm10-REAR}

\usepackage{upref,amsmath,amssymb,amsthm,mathrsfs}
\usepackage{enumerate}
\usepackage[colorlinks=true,linkcolor=blue,citecolor=blue]{hyperref}

\usepackage[scaled]{helvet} 
\usepackage{courier} 
\usepackage[mathbf]{euler}
\usepackage{mllocal}

\numberwithin{equation}{section}
\begin{document}
\title{Sums of regular selfadjoint operators in Hilbert-$C^*$-modules}

\author{Matthias Lesch}
\address{Mathematisches Institut,
Universit\"at Bonn,
Endenicher Allee 60,
53115 Bonn,
Germany}
\email{ml@matthiaslesch.de, lesch@math.uni-bonn.de}
\urladdr{www.matthiaslesch.de, www.math.uni-bonn.de/people/lesch}
\thanks{Both authors were partially supported by the 
        Hausdorff Center for Mathematics, Bonn}

\author{Bram Mesland}
\address{Mathematisches Institut,
Universit\"at Bonn,
Endenicher Allee 60,
53115 Bonn,
Germany}
\email{mesland@math.uni-bonn.de}


\begin{abstract}
    We introduce a notion of weak anticommutativity for a pair $(S,T)$
    of self-adjoint regular operators in a Hilbert $C^{*}$-module $E$.
    We prove that the sum $S+T$ of such pairs is self-adjoint and
    regular on the intersection of their domains. A similar result
    then holds for the sum $S^{2}+T^{2}$ of the squares. We show that
    our definition is closely related to the Connes-Skandalis
    positivity criterion in $KK$-theory. As such we weaken a sufficient
    condition of Kucerovsky for representing the Kasparov product. Our
    proofs indicate that our conditions are close to optimal.
\end{abstract}

\maketitle

\tableofcontents


\newcommand{\Cf}{\textit{Cf.~}}
\renewcommand{\cf}{\textit{cf.~}}

\newcommand{\Dom}[1]{\sD(#1)}
\newcommand{\LDom}[1]{\sD\bigl(#1\bigr)}
\newcommand{\NDom}{\sD}
\newcommand{\domS}{\sD(S)}
\newcommand{\domT}{\sD(T)}

\newcommand{\refone}{\textup{(1)}}
\newcommand{\reftwo}{\textup{(2)}}
\newcommand{\refthree}{\textup{(3)}}
\newcommand{\ctr}[1]{\bigl[ #1 \bigr] }
\newcommand{\Sl}{(S+\gl)\ii}
\newcommand{\Tl}{(T+\gl)\ii}
\newcommand{\Tm}{(T+\mu)\ii}

\newcommand{\dinn}[1]{\langle #1, #1 \rangle} 
\newcommand{\dlinn}[1]{\langle\langle #1 \rangle\rangle}

\newcommand{\Cltwo}{\C\ell(2)}
\newcommand{\haS}{\hat S}
\newcommand{\haT}{\hat T}

\newcommand{\mainclosedness}{\textup{(1)}}
\newcommand{\mainsymmetry}{\textup{(2)}}
\newcommand{\mainstrong}{\textup{(3)}}
\newcommand{\maincore}{\textup{(4)}}

\section{Introduction}
\label{s.intro}

A well-known problem in functional analysis is to describe
the domain and the spectral properties of the sum of two densely
defined closed operators.  In general nothing can be said as the
intersection of the domains can be just $\{0\}$. The problem
has a rich history and therefore in the next two sections we will summarize
what is known in two quite different contexts. Thereafter we will
describe the main theme of the paper.

\subsection{Banach space history of the problem}\label{ss.intro.BSH}
\begin{sloppypar}
Given two densely defined un\-bounded operators $A,B$ in some Banach
space $X$ with a joint ray, e.g. $i(0,\infty)$ or $(-\infty,0)$, in the
resolvent set. A basic problem is to give criteria which ensure the following
to hold:
\end{sloppypar}
\begin{thmenum}
\item $\ovl{A+B}+\gl$ is invertible for $-\gl$ in the said ray and large enough.
\item $A+B$ is a closed operator with domain $\Dom A \cap \Dom B$. 
\end{thmenum}
One of the first comprehensive papers on the problem
\cite{DPG1975} was motivated by evolution equations
\[
   \underbrace{-\pl_t^2}_{A} u + \underbrace{\Lambda(t)}_{B} u + \gl u = f,
\]
with $\Lambda(t)$ being a family of partial differential operators
parametrized by $t$.    

The validity of (1) means that the equation $A x + B x + \gl x = y$
is \emph{weakly} solvable for $\gl$ large, that is given $y$ there
is a sequence $x_n\in \Dom A\cap \Dom B$ such that $x_n\to x$ and 
$(A+B+\gl)x_n\to y$. (1) and (2) together mean that the equation
$A x + B x + \gl x = y$ is \emph{strongly} solvable for $\gl$
large, that is given $y$ there exists a solution $x\in\Dom A\cap \Dom B$.

One, and essentially the only approach to the problem in the
Banach space context rests on the idea of viewing $A+B+\gl$
as a (operator valued) function of $B$ and writing the
resolvent $(A+B+\gl)\ii$ as the Dunford integral
\begin{equation} \label{eq.intro.1}
    P_\gl := \frac{1}{2\pi i} \int_\Gamma (z+\gl+A)\ii \cdot (z-B)\ii dz,
\end{equation}
where $\Gamma$ is a suitable contour encircling the spectrum of
$B$. This approach works well only for \emph{sectorial} operators
with spectral angle $<\pi/2$. \Eqref{eq.intro.1} equals
the resolvent only if $A$ and $B$ are resolvent commuting and
so it is not surprising that in the literature certain commutator
conditions are formulated to ensure that \Eqref{eq.intro.1}
gives an appropriate approximation to the resolvent
\cite{DPG1975,DV1987,LT1987,Fuh1993,MP1997,KW2001,PS2007,Roi2016}.

\subsection{$KK$-theory history of the problem}
In the completely different context of $KK$-theory \cite{Kas:OKF} one encounters
the problem of regular sums of operators when one tries to construct the
notoriously complicated Kasparov product at the level of unbounded cycles
\cite{Mes2014,BMS2016,MR2016,KaaLes:LGP,KaaLes:SFU}. 

Here, the operators in
question act on a Hilbert-$(A,B)$-bimodule $E$, which is a complete inner product
module over the $C^*$-algebra $B$.  For an unbounded $B$-linear operator
$T$ in $E$ it makes sense to talk about self-adjointness and hence one might
be tempted to believe that everything is as nice as in a Hilbert space. This,
unfortunately (or fortunately), is not the case as the axiom of
\emph{regularity} does not come for free: analogously as in the Banach space
context above an unbounded self-adjoint $B$-linear operator $S$ in $E$ is
called regular if $S\pm \gl$ has dense range for one and hence for all
$\gl\in\C\setminus\R$. If $B=\C$ then regularity is equivalent to 
self-adjointness. In general, it is an additional feature, \cf
\cite{BaaJul:TBK, Wor:UEA, Pierrot2007, KaaLes:LGP}.

An \emph{unbounded Kasparov module} is a triple $(\sA, E, D)$ consisting of a Hilbert 
$(A,B)$-
bimodule $E$ and a self-adjoint regular operator $D$, with compact resolvent, that commutes 
with the dense subalgebra $\sA\subset A$ \emph{up to bounded operators}. 
In the construction of the tensor product of two such modules $(\sA,X,S_{X})$ 
and $(\sB, Y, T_{Y})$ one encounters two problems.

The first one is the definition of the operator $T=1\otimes_{\nabla} T_{Y}$ on 
the module $E:=X\otimes_{B}Y$. Since $T$ does not commute with $B$, one needs to 
incorporate extra data in the form of a \emph{connection} $\nabla$. This is discussed 
in great generality in \cite{MR2016} and in this paper we will not be concerned with 
this construction.

Once a well-defined self-adjoint and regular connection operator $T$ on $E$ 
has been constructed from$T_{Y}$, the second problem that needs to be addressed is
self-adjointness and regularity of the sum $D=S+T$, where $S=S_{X}\otimes 1$.
The goal is then to formulate an appropriate smallness condition
on the graded commutator $ST+TS$ such that $S+T$ is self-adjoint and
regular on $\Dom S\cap \Dom T$.

The Banach space results mentioned in
the previous paragraph do not (at least not a priori) apply to this
situation as in general self-adjoint operators are sectorial with spectral angle
$\pi/2$. Hence the sum of the spectral angles of $S$ and $T$ is $\pi$
which is exactly the threshold for the validity of regularity results
for sectorial operators, \cf Theorem \ref{p.CBSR.1} below. The methods in the Hilbert module case therefore resemble much more
the methods known from Hilbert space theory. 

\subsection{The main results}\label{ss.intro.MR}

Here we offer the following result which contains all previously known
results in this context as special cases \cite{Mes2014,KaaLes:LGP,MR2016}.

\begin{theorem}\label{p.intro.1}
Let $S,T$ be self-adjoint and regular operators in the Hilbert-$B$-module $E$.
Assume that 
\begin{thmenum}
\item there are constants $C_0,C_1,C_2>0$ such that the
form estimate
\begin{equation}
\label{eq.intro.2}
\inn{ [S,T] x, [S,T]x } \le C_0 \cdot \inn{ x,x }
            +  C_1 \cdot \inn{ Sx,Sx } + C_2 \cdot \inn{ Tx,Tx }
\end{equation}
holds for all $x\in \sF := \sF(S,T) = \bigsetdef{x\in\domS\cap\domT}{%
                              Sx\in\domT, Tx\in \domS}$.
This is an inequality in the $C^*$-algebra $B$.
\item  There is a core $\sE\subset \domT$ such that
$ (S+\gl)\ii \bl \sE \br \subset \sF(S,T)$ for $\gl\in i\R$, $|\gl|\ge
\gl_0$.
\end{thmenum}
Then $S+T$ is self-adjoint and regular on $\Dom S\cap \Dom T$.
That is for $z\in\C\setminus \R$ and $y\in E$ the equation
\[
     Sx+Tx+ z\cdot x = y
\]
has a unique (strong) solution $x\in\Dom S\cap\Dom T$. 


\end{theorem}

A more elaborate formulation can be found below in Theorem \ref{p.main.7}.
Our main application of Theorem \ref{p.main.7} is to the calculation of the
Kasparov product of unbounded cycles in $KK$-theory. 

Historically, the main tool for handling the Kasparov product has consisted of
a guess-and-check procedure pioneered by Connes-Skandalis \cite{ConSka}, and
later refined by Kucerovsky \cite{Kuc:KKP}.  This entails checking a set of
three sufficient conditions to determine whether a cycle $(\sA,E,D)$ is the
product of the cycles $(\sA, X,S_{X})$ and $(\sB, Y, T_{Y})$.  Although this
avoids the aforementioned hard problems, it leaves one with the burden of
coming up with a good guess for $D$ in every particular instance, as well as
proving that $(\sA, E,D)$ is a cycle.

In recent years, significant progress has been made on the constructive
approach to finding $D$. In this setting, the first sufficient condition of
Kucerovsky is satisfied whenever $D=S+T$ and $T$ is a connection operator
relative to $T_{Y}$. The second condition will be satisfied whenever
$\Dom{S+T}\subset\domS$. In previous work the condition 
\[
    \dinn{ [S,T]x }\leq C(\dinn{ x } +\dinn{ Sx }),
\]
was imposed to ensure self-adjointness of the sum $S+T$. This condition
implies that 
\[
    \inn{ (S+T)x,Sx}+\inn{ Sx,(S+T)x} \geq -\kappa\dinn{ x },
\]
for some $\kappa>0$, which is the third sufficient condition appearing in
\cite[Theorem 13]{Kuc:KKP}.  The form estimate \eqref{eq.intro.2} is in
general not compatible with Kucerovsky's estimate. In Section \ref{s.KK} we
prove that it is nonetheless sufficient to construct the Kasparov product.



\begin{theorem} \label{p.intro.2}
Let $(\sA, X,S_{X})$ and $(\sB, Y,T_{Y})$ be unbounded Kasparov modules for
$(A,B)$ and $(B,C)$ respectively and let $E:=X\otimes_{B}Y$ and
$S:=S_{X}\otimes 1$. Suppose that $T:\domT\to E$ is an odd self-adjoint
regular connection operator for $T_{Y}$ such that
\begin{enumerate}
\item for all $a\in\sA$ we have $a:\domT\to\domT$ and $[T,a]\in\sL(E)$;
\item $(S,T)$ is a weakly anticommuting pair.
\end{enumerate}
Then $(\sA,E,S+T)$ is an unbounded Kasparov module that represents the Kasparov 
product of $(X,S_{X})$ and $(Y,T_{Y})$.
\end{theorem}
We note that the statement that the sum operator $D=S+T$ is a $KK$-cycle is
part of this result.  The proof consists of showing that weak anticommutation
implies a weakened version of the sufficient conditions of Connes-Skandalis.
In the constructive setting, this supersedes the result of Kucerovsky and
covers a wider range of examples, provided that we construct our operator as a
sum.


\subsection{Outline}\label{ss.intro.Out}

The paper is organized as follows: 
In Section \ref{s.main} we first fix some notation and then introduce the
decisive notion of a weakly anticommuting pair of self-adjoint regular
operators. We put this definition into context and give a detailed comparison
to previous such notions.  Furthermore, by employing Clifford matrices we show
how to switch back and forth between weakly \emph{anticommuting} operators and
weakly \emph{commuting} operators.  Thereafter we formulate our main Theorem
on sums of self-adjoint regular operators followed by an outline of the
structure of the proof.

The proof of the main Theorem on sums is spread over the technical Sections
\ref{s.DC} and \ref{s.RA}. In Section \ref{s.appl} we provide some
applications. The squares $S^2$ and $T^2$ are sectorial operators with spectral angle $0$.
So they cry for a Dore-Venni type Theorem.  Although, they do not fulfill the
prerequisites for any of the Dore-Venni type Theorems we know of we can
nevertheless prove that $S^2+T^2$ is self-adjoint and regular on
$\Dom{S^2}\cap\Dom{T^2}$.  As another application we show that the sum Theorem
can be iterated to handle triples, and hence an arbitrary number of weakly
anticommuting summands. This is motivated by the second author's program of
constructing an appropriate category of $KK$-cycles. 

In the survey type Section \ref{s.CBS} we elaborate a bit more on the Banach
space approach to sums of operators and outline an alternative approach to our
main Theorem along the lines of the original Da Prato-Grisvard Dunford
integral \Eqref{eq.intro.1}, the main result being Theorem \ref{p.CBSR.4}.

The details on Theorem \ref{p.intro.2} can be found in 
Section \ref{s.KK}. Finally, the appendix contains a few useful commutator
identities which are needed in the proofs.

\subsection{Acknowledgment} The project started with discussions during the
Hausdorff Trimester programme ``Non-commutative Geometry and its
Applications'' which took place at the Hausdorff Research Institute for
Mathematics in Bonn, September 1 - December 19, 2014. Both authors gratefully
acknowledge the hospitality of this institution.

We also would like to thank Magnus Goffeng for his constant interest in this work,
in particular for the invitation to give lecture series on the
topic during the Master Class ``Sums of self-adjoint operators: Kasparov
products and applications'' at the University of Copenhagen, August 22--26,
2016. Furthermore ML would like to thank Elmar Schrohe for invitations to
Hannover and for his input on the Banach space aspects of the problem.

We thank Koen van den Dungen and Adam Rennie for motivating discussions
on the topic. We thank the anonymous referee for their careful reading of the manuscript.

\section{Weakly anticommuting operators and sums}
\label{s.main}
\subsection{Notation}\label{ss.main.not} We assume familarity with $C^*$-algebras, 
Hilbert-$C^*$-modules and unbounded and regular operators in Hilbert-$C^*$-modules
\cite{Lan:HCM,KaaLes:LGP}.

In the sequel $E$ will always be a Hilbert-$C^*$-module over the $C^*$-algebra
$B$.  By $\sL(E)$ we denote the $C^*$-algebra of bounded adjointable module
endomorphisms. By $S,T$ we denote self-adjoint regular operators in $E$. Domains
of (semi)regular operators are denoted by $\Dom{\ldots}$. Note that these are
always dense submodules.  Unless otherwise said, $\gl,\mu$ denote resolvent
parameters which are purely imaginary but bounded below by some $\gl_0>0$; the
specific value of $\gl_0$ is irrelevant and may vary from statement to
statement.  In norm estimates $C_1, C_2,\ldots$ denote generic constants; they
may also vary from statement to statement.

The basic problem we address is: if $ST+TS$ is ``small'' then $S+T$ should be
self-adjoint and regular. 

In the context of sectorial operators in certain Banach spaces this is a well
studied problem with numerous publications, e.g.  \cite{DPG1975,DV1987,PS2007}
and the references therein.

\subsection{Weakly anticommuting operators}
\label{ss.WAO}
For a pair of operators $S,T$ in a Hilbert-$B$-module $E$, we denote by
$[S,T]$ the \emph{anticommutator} $ST+TS$. This is in line with conventions
regarding graded Hilbert-$C^{*}$-modules and graded commutators in case the
operators $S$ and $T$ are both odd for the grading. Note that if either $S$ or
$T$ is not everywhere defined, then neither is the anticommutator $[S,T]$.

However, in order to work with commutators and anti-commutators at the same
time, it will be convenient to also use the notation
\[
   [S,T]_{\pm}=ST\pm TS=\pm [T,S]_{\pm}.
\]
Moreover, we will use the convention $[S,T]=[S,T]_+$. Some commutator
identities are collected in the Appendix \ref{s.A}.

\begin{dfn}\label{p.main.1}
Let $S,T$ be self-adjoint and regular operators in the Hilbert-$B$-module $E$
and set 
\begin{equation}\label{eq.main.2}
    \sF := \sF(S,T) = \bigsetdef{x\in\domS\cap\domT}{%
                              Sx\in\domT, Tx\in \domS}.
\end{equation}
The pair $(S,T)$ is called \emph{weakly anticommuting} if 
\begin{thmenum}
\item there are constants $C_0,C_1,C_2>0$ such that for all $x\in \sF$ the
form estimate
\begin{equation}\label{eq.main.1}
\inn{ [S,T] x, [S,T]x } \le C_0 \cdot \inn{ x,x }
            +  C_1 \cdot \inn{ Sx,Sx } + C_2 \cdot \inn{ Tx,Tx }
\end{equation}
holds in $B$.
\item  There is a core $\sE\subset \domT$ such that
$ (S+\gl)\ii \bl \sE \br \subset \sF(S,T)$ for $\gl\in i\R$, $|\gl|\ge
\gl_0$.
\end{thmenum}
The pair is called \emph{weakly commuting} if the estimate \Eqref{eq.main.1}
holds with the commutator $ST-TS$ instead of the anticommutator
$[S,T]$.
\end{dfn}
\begin{remark}\label{p.main.3}
\sitem This notion of weak (anti)commutativity is slightly more general than
corresponding notions in \cite[Assumption 7.1]{KaaLes:LGP},
\cite[Appendix A]{MR2016}, \cite[Sec. 3]{Les2016}. \Cf Section
\ref{ss.CPNWA} below.

\sitem One should also compare weak (anti)commutativity
to the commutator conditions appearing in earlier Banach space
literature on sums of operators; see Sections \ref{ss.intro.BSH}
and \ref{s.CBS}.

\sitem By definition 
\[
\sF(S,T) = (S+\gl)\ii \bl \domT \br \cap (T+\gl)\ii\bl\Dom{S}\br.
\]  If \reftwo\
holds with $\sE=\domT$ then this implies the equality
\begin{equation}\label{eq.main.3}
\begin{split}
  \sF(S,T) & = (S+\gl)\ii \bl \domT \br  \\
           & = \ran \bl (S+\gl)\ii\cdot (T+\mu)\ii \br,
\end{split}
\quad \gl,\mu\in i\R, |\gl|,|\mu|\ge \gl_0.
\end{equation}
By Theorem \plref{p.main.7} below indeed
\reftwo\ does hold with $\domT$ instead of $\sE$ as well.

\sitem 
It follows immediately from \reftwo\ of the definition that
$\sF(S,T)$ is a dense submodule of $E$. What is not immediately
obvious but will be proved below in Cor. \ref{p.DC.6}
is that $\sF(S,T)$ is a core for $S$ as well as for $T$. Theorem
\ref{p.main.7} below says even more. Namely, $\sF(S,T)$ is
a \emph{joint} core for $S$ and $T$ in the sense that for
$x\in\domS\cap\domT$ there is a sequence $x_n$ such that
$x_n\to x, S x_n\to Sx, Tx_n\to Tx$. There are concrete
formulas for the construction of $x_n$.


\sitem Self-adjointness implies that for $\gl\in i\R$ the operator $S+\gl$ is
bounded below by $|\gl|$. Hence the estimate \Eqref{eq.main.1} in the
definition implies for $x\in\sF(S,T)$ and $\gl,\mu\in i\R$ with
$|\gl|,|\mu|\ge \gl_0>0$ we have
\begin{equation}\label{eq.main.4}
\begin{split}
        \| [S,T]x \| &\le C \cdot (\|(S+\gl)x\| +  \| (T+\mu)x\| )\\
       & \le \frac{C}{|\mu|}\cdot \|(T+\mu)(S+\gl)x\| +
             \frac{C}{|\gl|}\cdot \|(S+\gl)(T+\mu)x\|.
\end{split}
\end{equation}
\end{remark}

\subsection{Comparison to previous notions of weak (anti)commutativity}
\label{ss.CPNWA}

We show that weak \emph{anti}commutativity in the sense of \cite[Appendix
A]{MR2016} implies weak \emph{anti}commutativity in the sense of
Def.~\ref{p.main.1}.  Similarly,  \cite[Assumption 7.1]{KaaLes:LGP} implies
weak \emph{commutativity} in the sense of Def.~\ref{p.main.1}. We denote by
$\tau\in \big\{+,-\big\}$ a fixed choice of sign.

\begin{prop}\label{p.main.4} Suppose that for all $\gl\in i\R, |\gl|\ge
\gl_0>0$ large enough
\begin{thmenum}
  \item there is a core $\sE\subset \domT$ for $T$ such that
  $(S+\gl)\ii(\sE) \subset \domS\cap\domT$,
\item $T(S+\gl)\ii (\sE) \subset \domS$, 
\item $[S,T]_{\tau} (S+\gl)\ii$ extends by continuity to
    a bounded (adjointable) map $E\to E$.
\end{thmenum}
Then $S,T$ are weakly anticommuting (if $\tau=+$) resp. weakly commuting (if
$\tau=-$) in the sense of Def. \ref{p.main.1} with the constant $C_2$ in
\Eqref{eq.main.1} being $0$.
\end{prop}

\begin{remark}\label{p.main.5}
In \refthree\ we put ``adjointable'' in parentheses because
we do not have to assume this. Rather it follows because if $[S,T]_\tau \Sl$
is a bounded module map $E\to E$ for $|\gl|\ge \gl_0$ then its adjoint
is given by $(S+\ovl \gl)\ii [T,S]_\tau$, which turns out to be bounded as
well.
\end{remark}

\begin{proof}\sitem Actually in this case it is easy to show a priori that the
core $\sE$  can be replaced by $\domT$: let $x\in\domT$ and let
$(x_n)_n\subset \sE$ be a sequence with $x_n \to x$ and $Tx_n\to Tx$. Then by
\refone\ we have $(S+\gl)\ii x_n \in\domT$ and by continuity 
$(S+\gl)\ii x_n \to  (S+\gl)\ii x$. 
By \reftwo\ we have $T(S+\gl)\ii x_n\in \domS$ and by \refthree\
\[
    (S-\tau \gl) T (S+\gl)\ii x_n =
            (ST+\tau \cdot TS) (S+\gl)\ii x_n -\tau \cdot Tx_n
\]
converges as well. This shows that $T(S+\gl)\ii x_n$ converges in $\domS$ and
thus $(S+\gl)\ii x\in\domT$ and $T(S+\gl)\ii x\in\domS$. This proves that
\refone\ and \reftwo\ hold for $\domT$ instead of $\sE$.

To see that also \refthree\ holds for $\domT$ instead of $\sE$ we need to show
that the continuous extension of 
$\bl (ST+\tau\cdot TS)(S+\gl)\ii \br{\big|_\sE}$ to $E$ coincides with the now
defined operator $\bl (ST+\tau\cdot TS)(S+\gl)\ii \br{\big|_{\domT}}$.  We
already know that $\bl T(S+\gl)\ii x_n \br_n$ converges in $\domS$ and from
\refthree\ we know that $\bl (ST+\tau\cdot TS)(S+\gl)\ii x_n \br_n$ converges;
thus also $\bl TS(S+\gl)\ii x_n\br_n$ converges. Summing up we have that
$T(S+\gl)\ii x\in\domS, S(T+\gl)\ii x\in\domT$ and hence 
\[
(ST+\tau\cdot TS)(S+\gl)\ii x = \lim_n (ST+\tau\cdot TS)(S+\gl)\ii x_n
\]
as claimed.

\sitem The first part of this proof
shows that for $x\in\domT$ we have 
\[
  (S+\gl)\ii x\in 
     \domS\cap\domT, S(S+\gl)\ii x\in\domT, T(S+\gl)\ii x\in\domS,
\]
thus $(S+\gl)\ii\bl \domT\br \subset \sF(S,T)$, for $\gl\in i\R,
|\gl|\ge \gl_0$. In fact equality holds: namely, given $y\in
\sF(S,T)$ then $x :=(S+\gl) y\in\domT$ and hence $y = (S+\gl)\ii x\in\sF$.

On $\sF = \ran\Bl (S+\gl)\ii \cdot (T+\mu)\ii \Br$ we now have
for a fixed $|\gl|\ge \gl_0$, using \refthree, 
\begin{equation}\label{eq.main.5}
\begin{split}
     \inn{ [S,T]_{\tau}y,  [S,T]_{\tau}y } 
        & =  \inn{ [S,T]_{\tau}(S+\gl)\ii(S+\gl)y,  [S,T]_{\tau}(S+\gl)\ii (S+\gl)y }  \\
        & \le   C \cdot \inn{ (S+\gl) y,(S+\gl) y }  
          \le  C_0\cdot \inn{ y,y } + C_1 \cdot \inn{ Sy,Sy  }, 
\end{split}
\end{equation}
and the result follows.
\end{proof}

\begin{remark}\label{p.main.6}
For $y\in\sF(S,T)$, the form estimate
\begin{equation}\label{eq.main.6}
    \inn{[S,T]_{\tau}y,  [S,T]_{\tau}y } \le  C_1\cdot \inn{ y,y } + C_2 \cdot \inn{ Sy,Sy  }, 
\end{equation}
implies the norm estimate
\begin{equation}\label{eq.main.7}
    \|[S,T]_{\tau}y\|\le C_0\cdot \|y\| + C_1 \cdot \|Sy\|,
\end{equation}
but is in fact equivalent to it, provided that $(S+\lambda)^{-1}\Dom{T}\subset \sF(S,T)$. This is seen by observing that
\Eqref{eq.main.7} implies that
\[
    [S,T]_{\tau}(S+\gl)\ii :(S+\gl)\sF(S,T)\to E,
\]
extends to a bounded adjointable operator. Then by  writing
\[
    [S,T]_{\tau}y=[S,T]_{\tau}(S+\gl)\ii (S+\gl)y,
\]
and applying the standard form estimate 
$\inn{ Ry,Ry } \leq \|R\|^2\inn{ y,y }$ for adjointable operators we obtain
\eqref{eq.main.6}. Of course a similar equivalence holds when we exchange $S$
and $T$. It should be noted that when both $C_{1}$ and $C_2$ are nonzero, we
cannot a priori replace \eqref{eq.main.1} by the corresponding norm estimate.
For that we would need the regularity of the operator $D=S+T$, as the above
argument works by using the operator $(1+D^{*}D)^{-\frac{1}{2}}$.
\end{remark}

\subsection{Clifford algebras and (anti)commutators}
\label{ss.Clifford}

We briefly explain how one can switch between commuting and anticommuting
operators using Clifford algebra identities.

Let $\Cltwo$ be the complex Clifford algebra on two unitary
self-adjoint generators $\sigma_{i}$, $i=1,2,$ satisfying the
relations $\sigma_{i}\sigma_{j} + \sigma_{j}\sigma_{i} = 2\delta_{ij}$.
In fact $\Cltwo\simeq M(2,\C)$ with generators given by
\footnote{$\sigma_3$ is the volume element of $\Cltwo$, alternatively
$\sigma_1, \sigma_2,\sigma_3$ generate one of the two irreducible
representations of $\C\ell(3) \simeq M(2,\C)\oplus M(2,\C)$.}
\[
\sigma_{1}:= \begin{pmatrix} 0 & 1 \\ 1 & 0 \end{pmatrix},\quad
\sigma_{2}:= \begin{pmatrix} 0 & i \\ -i & 0 \end{pmatrix},\quad
\sigma_{3} := i \,\sigma_{1} \sigma_{2} = \begin{pmatrix} 1 & 0 \\ 0 & -1\end{pmatrix}.
\]
Given operators $S,T$ on the Hilbert-$B$-module $E$ let 
$\hat E = E\otimes \C^2=E\oplus E$ and  consider the operators $\hat S$ and
$\hat T$ on $E\oplus E$ given by 
\begin{align}
\label{Shat}
\Dom{\hat S} &:= \domS\oplus \domS,
                  &  \hat S &= S\otimes I = \begin{pmatrix} S & 0 \\ 
                                                            0 & S
                                            \end{pmatrix},  \\
\label{That}
\Dom{\hat T} &:= \domT\oplus \domT, 
                  &  \hat T & = T\otimes I = \begin{pmatrix} T & 0 \\
                                                             0 & T
                                             \end{pmatrix}.
\end{align}

The $C^{*}$-algebra $\Cltwo\subset\sL(\hat E)$ is represented unitarily on
$\hat E$. The submodules $\Dom{\hat S}$ and $\Dom{\hat T}$ are $\Cltwo$
invariant, and the representation commutes with $\hat S, \hat T$, so that the
operators 
\[
         s_{i}:=\hat S\sigma_{i},\quad \Dom{s_{i}}:=\Dom{\hat S},
   \quad t_{j}:=\hat T\sigma_{j},\quad \Dom{t_j}:=\Dom{\hat T},
\] 
are all self-adjoint and regular in $\hat E$.  It then holds that
\begin{align*}
    \sF (s_{i},t_{j})
            :&=\bigsetdef{  x\in \Dom{s_{i}} \cap \Dom{t_{j}}}{%
                 s_i x\in \Dom{t_{j}}, t_jx\in \Dom{s_{i}}}\\
             &=\sF(S,T)\oplus \sF(S,T),
\end{align*}
and we have the following relations:
\begin{align}\label{eq.main.8}
   \hat S \sigma_{i} \cdot \hat T \sigma_{j} \pm \hat T \sigma_{j} \cdot \hat S \sigma_{i} 
            &= (\hat S\hat T\mp\hat T\hat S)\sigma_{i}\sigma_{j},\quad i\neq j,  \\
\nonumber    (\hat S\sigma_{i})^2 &= \hat S^2,\quad (\hat T\sigma_{i})^2 = \hat T^2,\\
\label{eq.main.9}  (\hat S\sigma_{i}+\gl)\ii  = (\hat S\sigma_{i} -\gl)(\hat S^2-\gl^2)\ii 
                        &= (\hat S-\gl)\ii \sigma_{i}-
                             (\gl \sigma_{i} +\gl) (\hat S^2-\gl^2)\ii .
\end{align}

It follows from \Eqref{eq.main.9} that for all $i,j$ we have
\[
   (s_{i}+\gl)\ii \Dom{t_{j}}\subset\sF ( s_i,t_{j}),
\]
and then from \Eqref{eq.main.8} that for $i\neq j$, the pair
$(s_{i},t_{j})$ is weakly commuting whenever $(S,T)$ is weakly anticommuting
and vice versa. Out of a pair of weakly anticommuting operators we so obtain
three pairs $(s_{i},t_{j}), (i\neq j)$ of weakly commuting operators and
similarly so for a pair of weakly commuting operators.

\subsection{The Main Theorem}

\begin{theorem}\label{p.main.7} Let $S,T$ be weakly anticommuting operators in
the Hilbert-$B$-module $E$. Then the operator $S+T$ is self-adjoint and
regular on $\domS\cap\domT$. In more detail we have the following:

\begin{thmenum}
\item There is a constant $C$ such that for $x\in\domS\cap\domT$ we have
\begin{equation}
    \begin{split}
    C\ii  \cdot \bl \inn{ x, x } + \inn{ (S+ T)x,(S+T)x } \br
          & \le  \inn{ x,x } + \inn{ Sx,Sx }+\inn{ Tx,Tx } \\ 
          &\le C \cdot \bl \inn{ x, x } + \inn{ (S+ T)x,(S+T)x } \br.
    \end{split}
    \label{eq.main.10}
\end{equation}

\item For $\gl,\mu\in i\R, |\gl|,|\mu|\ge \gl_0$ large enough we have
\begin{equation}\label{eq.main.11}
    (T+\mu)\ii\bl \domS \br = \sF(S,T) = (S+\gl)\ii\bl \domT \br
\end{equation}
and hence
\begin{equation}\label{eq.main.12}
    \ran\bl (T+\mu)\ii\cdot (S+\gl)\ii \br = \sF(S,T) = 
          \ran\bl (S+\gl)\ii\cdot (T+\mu)\ii \br.
\end{equation}
\item For $\gl_0>0$ large enough and $\gl,\mu\in i\R, |\gl|>|\mu|\ge \gl_0,
\mu/\gl>0$ the operator
\[
   S+T+\frac{TS}{\gl} +\mu: \sF\to E
\]
is bijective and its inverse $(S+T+\frac{TS}{\gl} +\mu)\ii $ is a bounded
adjointable operator. Moreover, for fixed $\mu$ and for all $x\in E$ 
\begin{equation}\label{eq.main.13}
     \lim_{|\gl|\to \infty}(S+T+\mu)(S+T+\frac{TS}{\gl} +\mu)\ii x=x,
\end{equation}
in norm and
\begin{equation*}\lim_{|\gl|\to \infty}(S+T+\frac{TS}{\gl} +\mu)\ii =(S+T+\mu)\ii ,
\end{equation*}
in operator norm.\footnote{The limit $\lim_{|\gl|\to \infty}$ is taken for
the net of those $\gl$ with $|\gl|>|\mu|, \mu/\gl>0$. In the sequel this is
understood without repeatedly mentioning it.}

\item $\sF(S,T)$ is a core for $S,T,$ and $S+T$.
\end{thmenum}
\end{theorem}

\begin{remark}\label{p.main.8}
Item \maincore\ can be made more precise.
For $x\in\domS\cap\domT$ it follows from \mainstrong\ that
\[
    x_\gl := (S+T+ \frac{TS}{\gl}+\mu)\ii (S+T+\mu)x
\]
converges, as $|\gl|\to\infty$, to $x$ in the graph norm of $S+T$.
By \mainclosedness\ this means that $x_\gl\to x, Sx_\gl\to Sx, Tx_\gl\to Tx$.

Alternatively, the method of proof of \mainstrong\ can be used
to show the following slightly stronger convergence result:
for $x\in E$ put
\begin{equation}\label{eq.main.14}
    x_\gl := \gl^2\cdot (T+\gl)\ii\cdot(S+\gl)\ii x\in\sF(S,T), 
\end{equation}
for $\gl\in i\R, |\gl|>\gl_0$. Then $\lim\limits_{|\gl|\to\infty} x_\gl = x$.
Furthermore, if $x\in \domS$ (resp. $\domT$) then also
$\lim\limits_{|\gl|\to\infty} S x_\gl = S x$ (resp. 
$\lim\limits_{|\gl|\to\infty} T x_\gl = T x$).
\end{remark}

The proof of Theorem \ref{p.main.7} will be broken down as follows:

\resetsitem

\sitem First we prove the form estimate \Eqref{eq.main.10}, as a rather direct
consequence of the form estimate \Eqref{eq.main.1}. From \Eqref{eq.main.10} we
derive that the operator $S+T$ is closed on $\domS\cap\domT$.

\sitem Next we will show \mainsymmetry\ which shows, among other things, that
a posteriori \reftwo\ of Def.~\ref{p.main.1} holds for $\domT$ instead of
the core $\sE$ and that the roles of $S,T$, which a priori appear in
Def.~\ref{p.main.1} \reftwo\ in an unsymmetric way, can be reversed. See
Section \ref{s.DC}.

\sitem For $\gl\in i\R$ with $|\gl|\geq \gl_0$, the
operators \[A_\gl :=S+T+\frac{TS}{\gl}:\sF\to E,\] are well defined.
We prove a fundamental lower form bound for the operators $A_\gl +\mu$.
This allows to deduce $\gl$-uniform norm bounds on the bounded adjointable
inverse of these operators, and that for fixed $\mu$ the net
$\bl (A_\gl +\mu)\ii  \br_\gl$ is operator norm Cauchy in $\gl$.

\sitem Subsequently we show that the convergence in \Eqref{eq.main.13} holds true
for all $x\in E$.  From that we deduce directly that the operators $S+T+\mu$
have dense range, and therefore are essentially self-adjoint and regular on
$\sF(S,T)$. It is then readily established that the limit of the Cauchy net
$\bl (A_\gl +\mu)\ii \br_\gl$ is in fact $(S+T+\mu)\ii $, the resolvent of the sum
operator.

\sitem In the weakly commuting case, we obtain that the operators 
\[
\begin{pmatrix} 0 & S+iT\\ 
                S-iT & 0
\end{pmatrix},\quad \begin{pmatrix} S & T \\ T & -S\end{pmatrix},\quad \textnormal{and}\quad \begin{pmatrix} S & iT \\ -iT & -S\end{pmatrix},
\]
are self-adjoint and regular on $\bl\domS\cap\domT\br^{\oplus 2}$.
In particular the operators $S\pm i T$ on
$\domS\cap\domT$ are closed, regular and adjoints of each other.
\resetsitem

\section{Domain considerations: a closer look at $\sF(S,T)$}
\label{s.DC}

Estimates for inner products in Hilbert-$C^*$-modules are a little more
delicate since one is dealing with inequalities in $C^*$--algebras.  Recall
that for $x,y\in E$ one has
\begin{equation}\label{eq.DC.1}
   \inn{x,y} + \inn{y,x} \le \inn{x,x} + \inn{y,y},
\end{equation}
which follows immediately from expanding $0\le \inn{x-y,x-y}$.
By replacing $x$ by $\pm r^{\frac{1}{2}}\cdot x$ and $y$ by $r^{-\frac{1}{2}} \cdot y$ we obtain
for \emph{any} $r>0$ and $x,y\in E$
\begin{equation}\label{eq.DC.2}
   \pm\bl \inn{x,y} + \inn{y,x} \br 
        \le r\cdot \inn{x,x} + \frac 1r\cdot \inn{y,y}.
\end{equation}
Note that this is an inequality from above and from below. This estimate
will be used repeatedly in the sequel.

\subsection{Graph norm estimate} 
\label{ss.PI.2}
We now show that the inequalities \Eqref{eq.main.10} hold true for
$x\in\sF(S,T)\subset\domS\cap\domT$. 

\begin{prop}\label{p.DC.1} Let $S,T$ be weakly anticommuting operators in the
Hilbert-$C^{*}$-module $E$.  There is a constant $C>0$ such that for all
$x\in\sF(S,T)$ the form estimate 
\[ 
    \begin{split}
    C\ii  \cdot \bl \inn{ x, x } + \inn{ (S+ T)x,(S+T)x } \br
          & \le  \inn{ x,x } + \inn{ Sx,Sx }+\inn{ Tx,Tx } \\ 
          &\le C \cdot \bl \inn{ x, x } + \inn{ (S+ T)x,(S+T)x } \br
    \end{split}
\]
holds true.
\end{prop}

\begin{proof}
For any $x\in\domS\cap \domT$ we write
\begin{align*}
    \inn{ (S+T)x,(S+T)x }&=\inn{ Sx,Sx } +\inn{ Tx,Tx } +\inn{ Sx,Tx }+\inn{ Tx,Sx }\\
         &\leq 2\cdot (\inn{ Sx,Sx } +\inn{ Tx,Tx }). 
\end{align*}
To prove the lower estimate, if we furthermore assume that 
\[
    x\in\sF(S,T)\subset \domS\cap\domT
\]
we have for any $K>0$ (\cf \Eqref{eq.DC.2})
\[
    \pm \bl \inn{ Sx,Tx }+\inn{ Tx,Sx } \br
   =\pm \inn{ [S,T]x,x }\leq  K\ii \cdot \inn{ [S,T]x,[S,T]x } + K\cdot \inn{ x,x }.
\]
Thus for $C$ as in \Eqref{eq.main.1} and any $0<\eps<1$ 
we can choose $K>C$ such that $K\ii C<\eps$. We then find
\begin{align*}
    \inn{ [S,T]x,x } & \geq - K\ii \cdot \inn{ [S,T]x, [S,T]x  } - K \cdot \inn{ x,x }\\
       &\geq -\eps \cdot \bl \inn{ Sx, Sx }+\inn{ Tx, Tx } \br 
          - (K+\eps )\cdot \inn{ x,x },
\end{align*}
and thus 
\begin{multline*}
    \dinn{ Sx }+\dinn{ Tx } 
       \leq (1+\eps)\cdot \bl\inn{ Sx,Sx }+\inn{ Tx,Tx }\br +
               \inn{ [S,T]x,x } + (K+\eps)\cdot \inn{ x,x }  \\
        =  \inn{ (S+T)x, (S+T)x } + \eps\cdot \bl \dinn{Sx} + \dinn{Tx} \br
               + (K+\eps)\cdot \dinn{ x },
\end{multline*}
hence the estimate
\[
    (1-\eps)\cdot \bl \dinn{ Sx }+\dinn{ Tx }\br
       \leq \dinn{ (S+T)x } + (K+\eps)\cdot \dinn{ x }.
\]
With the constant
\[
  C :=\max\left\{\frac{1}{1-\eps},\frac{K+\eps}{1-\eps}\right\},
\]
we find
\[
\dinn{ Sx }+\dinn{ Tx }\leq C\cdot \bl \dinn{ (S+T)x } + \dinn{ x }\br,
\]
as desired.
\end{proof}

\subsection{Symmetry of the axioms for weak (anti)commutativity; proof of part
\textup{(2)} of the main Theorem \plref{p.main.7}}
\label{ss.DC.1}

We first note that for arbitrary resolvent parameters $\gl,\mu$ we have the
equality of commutators
\begin{equation}\label{eq.DC.3}
	[ S+\gl, T+\mu ]_- = [S,T]_-;
\end{equation}
there is no simple analogue to this in the weakly anticommuting case.

In this section we establish part \textup{(2)} of the main Theorem
\plref{p.main.7}. For the proof we need two preparatory Lemmas.

\begin{lemma}\label{p.DC.2}
Let $S,T$ be weakly commuting operators. Then for $\gl_0$ large
enough we have for all $\gl,\mu\in i\R, |\gl|,|\mu|\ge \gl_0$
and $x\in\sF(S,T)$
\begin{align}
\| [S+\gl,T+\mu]_- x\| = \|[S,T]_-x \| & \le C\cdot\bl
\frac1{|\gl|}+\frac1{|\mu|}\br\cdot \| (S+\gl)(T+\mu) x \|,\label{eq.DC.4}\\
\| [S+\gl,T+\mu]_- x\| = \|[S,T]_-x \| & \le C\cdot\bl
\frac1{|\gl|}+\frac1{|\mu|}\br\cdot \| (T+\mu)(S+\gl) x \|\label{eq.DC.5}.
\end{align}
By increasing $\gl_0$ we can arrange the constant 
$C\cdot \bl\frac1{|\gl|}+\frac1{|\mu|}\br \le 2\cdot  C\cdot \frac1{\gl_0}$ 
to be as small as we please.
\end{lemma}

\begin{proof}
By \Eqref{eq.DC.3} and \Eqref{eq.main.4} we have
\begin{align*}
\| [S+\gl,T+\mu]_- x\| &= \|[S,T]_-x \|  \le \frac{C_1}{|\mu|} \cdot \| (T+\mu)(S+\gl) x \|
                         + \frac{C_2}{|\gl|}\cdot \| (S+\gl)(T+\mu) x \| \\
    &\le \frac{C_1}{|\mu|} \cdot \| [T,S]_- x \| + \bl\frac{C_1}{|\mu|}+\frac{C_2}{|\gl|}\br
                               \| (S+\gl)(T+\mu) x \|.
\end{align*}
If $\gl_0$ is large enough such that $\frac{C_1}{\gl_0}<1$ we obtain estimate
\Eqref{eq.DC.4}. The proof of \Eqref{eq.DC.5} is completely analogous.
\end{proof}

\begin{lemma}\label{p.DC.3} Let $A,B$ be closed densely defined and boundedly
invertible operators in the Banach space $X$ and suppose that a subspace
\[
    \sE\subset\Dom{A}\cap\Dom{B}\subset X
\]
is a core for $A$. If there exists $0<\eps<1/3$ such that for all
$x\in\sE$

\begin{equation}\label{eq.DC.6}
	\|Ax - B x\| \le \eps \cdot \bl \|Ax\| + \|Bx\| \br,
\end{equation}
then $\Dom{A}=\Dom{B}$ and the estimate \Eqref{eq.DC.6}
as well as
\begin{equation}\label{eq.DC.7}
    \begin{split}
	\|Ax - B x\| \le \frac{2\eps}{1-\eps}\cdot \|Bx\|,\\
	\|Ax - B x\| \le \frac{2\eps}{1-\eps}\cdot \|Ax\|,
    \end{split}
\end{equation}
hold for all $x\in\Dom{A}=\Dom{B}$.
\end{lemma}

\begin{proof} Clearly, for $x\in\sE$ we infer from \Eqref{eq.DC.6} that
\begin{equation}\label{eq.DC.8}
     \|Ax - Bx\| \le 2 \eps\cdot \|Ax\| + \eps\cdot \|Ax-Bx\|,
\end{equation}
hence the first inequality in \Eqref{eq.DC.7} for $x\in\sE$;
the second is seen by exchanging $A$ and $B$. Thus for $x$ in the dense subspace $A(\sE)\subset X$ we have
\begin{align*}
   \| B A\ii x - x \| & = \| B (A\ii x) - A (A\ii x) \|\\
                      & \le 2\, \eps\cdot \|x\| + \eps\cdot \| B A\ii x - x \|,
\end{align*}
hence 
\begin{equation}\label{eq.DC.9}
 \| BA\ii x - x\| \le \frac{2\eps}{1-\eps} \cdot\|x\|; \quad \frac{2\eps}{1-\eps} < 1. 
\end{equation}
If $x\in X$ let $(x_n)\in A(\sE)$ be a sequence with $x_n\to x$.
Then \Eqref{eq.DC.9} and the closedness of $B$ imply that $A\ii x\in\Dom{B}$
and \Eqref{eq.DC.9} holds for $x$ as well. Since $A\ii(X) = \Dom{A}$
it follows that $\Dom{A}\subset\Dom{B}$ and \Eqref{eq.DC.6}, \eqref{eq.DC.7}
hold for all $x\in\Dom{A}$. From \Eqref{eq.DC.9} it follows that
$BA\ii$ is bounded and invertible and thus $B$ maps
$\Dom{A}$ bijectively onto $X$. Since $B$ is assumed to be invertible
$\Dom{B}\to X$,
it follows that $\Dom{B}=B\ii(X) = \Dom{A}$.
\end{proof}

For future reference we note
\begin{cor}\label{p.DC.4} If under the assumptions of Lemma \plref{p.DC.3},
$\sX\subset X$ is a dense subspace then $A\ii(\sX)$ is a core
for $A$ and for $B$.
\end{cor}
\begin{proof} Since $A$ is boundedly invertible, $A\ii$ is a Banach
space isomorphism from $X$ onto $\dom(A)$, the latter being equipped
with the graph norm. Hence $A\ii$ maps dense subspaces of $X$ onto
dense subspaces of $\dom(A)$ and vice versa. \Eqref{eq.DC.7}
shows that the graph norms of $A$ and $B$ are equivalent,
thus cores for $A$ are cores for $B$ and vice versa.
\end{proof}

\begin{prop}\label{p.DC.5}
Let $S,T$ be weakly (anti)commuting operators. Then \reftwo\ of
Definition \plref{p.main.1} holds with $\sE=\domT$. That is
$(S+\gl)\ii\bl\domT\br\subset \sF(S,T)$ for $\gl\in i\R,
|\gl|\ge \gl_0$. Furthermore, we have for $\gl,\mu\in i\R,
|\gl|, |\mu|\ge \gl_0$
\[
    \sF(S,T)=\ran\bl (T+\mu)\ii\cdot (S+\gl)\ii \br = 
       \ran\bl (S+\gl)\ii\cdot (T+\mu)\ii \br,
\]
that is, \textup{(2)} of Theorem \plref{p.main.7} holds.
\end{prop}

\begin{proof}
The discussion in Section \ref{ss.Clifford} shows that the statement for
weakly commuting operators implies that for weakly anticommutating operators,
so it suffices to consider only this case.  Choosing $\gl_0$ in
\Eqref{eq.main.4} large enough we find
for $x\in\sF(S,T)$:
\begin{align*}
	\| (S+\gl)\cdot (T+\mu) x &- (T+\mu)\cdot(S+\gl) x\| = \| (ST-TS)x \| \\
             &  \le \eps\cdot \Bl \|(T+\mu)(S+\gl) x\| + \|(S+\gl)(T+\mu) x\| \Br,
\end{align*}
with $\eps$ as small as we please, e.g. $ < \frac 13$.

The result now follows from Lemma \ref{p.DC.3} with 
$A=(T+\mu)(S+\gl),$ and $B=(S+\gl)(T+\mu)$.
Note that $\sE$ of Definition \ref{p.main.1} is a core
for $T$ and hence 
$(S+\gl)\ii\bl \sE \br\subset \sF(S,T)\subset \Dom{B}$
is a core for $A$, \cf Cor. \ref{p.DC.4}.
\end{proof}

\begin{cor}\label{p.DC.6} 
Let $S,T$ be weakly (anti)commuting operators. Then 
$\sF(S,T)$ is a core for $S$ as well as a core for $T$.
\end{cor}
\begin{proof} For $\gl\in i\R, \gl\not=0$ the resolvent
$\Sl$ maps the dense subspace $\domT$ into a core
for $S$ and the resolvent $\Tl$ maps the dense subspace
$\domS$ into a core for $T$. By Prop. \ref{p.DC.5}
and Definition \ref{p.main.1} we have for $|\gl|$ large
enough $\Sl(\domT)\subset \sF(S,T)$ and
$\Tl(\domS) \subset \sF(S,T)$, hence the claim.
\end{proof}

\section{Approximation of the resolvent of the sum $S+T$}
\label{s.RA}

During the whole section let $S,T$ be a weakly anticommuting pair
of self-adjoint regular operators in the Hilbert-$B$-module $E$.

By Proposition \ref{p.DC.5}, we have the equality of submodules
\[
    \sF(S,T)=\ran  (S+\gl )\ii (T+\mu)\ii =\ran (T+\mu)\ii (S+\gl)\ii 
\]
for $\gl,\mu\in i\R, |\gl|, |\mu|\ge \gl_0$ large enough.
The operators $ST$ and $TS$ are defined on $\sF(S,T)\subset \domS\cap\domT$, 
which by Cor. \ref{p.DC.6} is a core for $S$ as well as for $T$.
We now consider the operator
\[
    A_\gl :\sF(S,T)\to E, \quad A_\gl x:=Sx+Tx+\frac{TSx}{\gl},
\]
which approximates $S+T$ strongly on $\sF(S,T)$. In this section we show that
this approximation holds in a much stronger sense.

\subsection{The fundamental estimate}

\begin{lemma}\label{p.RA.1} 
  For $\gl, \mu\in i\R, |\gl-\mu| < |\gl|$, the operator 
$A_\gl +\mu: \sF(S,T)\to E$ is bijective and hence boundedly invertible.
\end{lemma}

\begin{proof} Recall that for a closable operator in a Banach space
it is a consequence of the Closed Graph Theorem that being bijective is
equivalent to being closed \emph{and} boundedly invertible.
We have 
\begin{equation}\label{eq.RA.1} 
    \begin{split}
        A_\gl + \mu & = \gl\ii (T+\gl)(S+\gl) + \mu -\gl =: B_\gl +\mu-\gl.
    \end{split}
\end{equation}

For $\gl=\mu$ large enough this operator is boundedly invertible 
by Proposition \ref{p.DC.5}. Moreover, since $\gl, \mu\in i\R$ and $S, T$
are self-adjoint, we have $\|(A_\gl+\gl)\ii\|\le \frac 1{|\gl|}$. Thus
\[
\begin{split}
    \bigl\| (A_\gl+\mu)(A_\gl+\gl)\ii - \Id \bigr\|
      & = |\gl-\mu|\cdot \|(A_\gl+\gl)\ii\| \le \frac{|\gl-\mu|}{|\gl|}<1,
\end{split}
\]
hence the claim.
\end{proof}


It is immediately clear that for $x\in\sF(S,T)$ and $|\gl|\to \infty$ we have
the convergence
\[
    (A_\gl +\mu)x\to (S+T+\mu)x.
\]
We will show that this convergence holds in the resolvent sense. We now
proceed to derive the form estimate that provides the cornerstone of our
argument. 
\begin{lemma}\label{p.RA.2}
For $\gl\in i\R\setminus\{0\}$, $|\gl|\ge \gl_0$ large enough
the operators  $A_\gl :=S+T+ \frac{TS}{\gl}$ are closed on the domain 
\[
  \Dom{ A_\gl  }:=\sF(S,T)
\]
and satisfy $\Dom{ A_\gl  }= \Dom{ A_\gl ^{ *} }$.
For $\gl\in i\R, |\gl|\ge\gl_0$ and $\mu\in i\R$ with $|\gl-\mu|<\gl$,
the operator $A_\gl+\mu$ is boundedly invertible. Moreover, there exist
$C, \mu_0>0$ such that for $\gl, \mu\in i\R$, $|\gl|>|\mu|\ge \mu_0,
\gl/\mu>0$ we have
\begin{equation}\label{eq.RA.2}
   \|(A_\gl +\mu )\ii \|\le \frac{\sqrt 2}{|\mu|},\quad
   \| S(A_\gl +\mu )\ii \|\le \sqrt 2, \quad 
   \| T(A_\gl +\mu )\ii \|\le \sqrt 2, 
\end{equation}
\begin{equation}\label{eq.RA.3}
   \|\frac{TS}{\gl}(A_\gl +\mu )\ii \| \le 1,\quad 
   \|[S,T](A_\gl +\mu )\ii \| \le C.
\end{equation}
\end{lemma}

\begin{proof}
Recall that $A_\gl+\gl = \frac 1\gl (T+\gl)(S+\gl)$. Hence
the claims about closedness, $\Dom{A_\gl} = \Dom{A_\gl^*}$
and bounded invertibility follow from Proposition \ref{p.DC.5}
and Lemma \ref{p.RA.1}. It remains to prove the
claimed estimates.

Keep in mind that $\overline{\gl}=-\gl$ and similarly for $\mu$, since these
numbers are purely imaginary.
Furthermore, we introduce the convenient abbreviation
\begin{equation}\label{eq.convenient-abbrev}
    \dlinn y := \inn{y,y}
\end{equation}
for $y\in E$.

For $x\in\sF(S,T)$ we first expand

\begin{align}
\nonumber\langle\langle (A_\gl  +\mu)  x& \rangle\rangle =
    \dlinn{ Sx }+\dlinn{ Tx }+\inn{ [S,T]x,x } +\inn{ (S+T)x, (\frac{TS}{\gl} + \mu)x } \\
   &\nonumber\qquad   +\inn{ (\frac{TS}{\gl } +\mu )x, (S+T)x }+ \dlinn{ (\frac{TS}{\gl } +\mu )x } \\
&\nonumber      =\dlinn{ Sx }+\dlinn{ Tx }+ \inn{ [S,T]x,x } +\frac{1}{\gl } \bl\inn{ Sx,TSx }-\inn{ TS x,Sx }\br \\
&\nonumber\quad    + \inn{ Tx,\frac{TS}{\gl }x }+\inn{ \frac{TS}{\gl }x, Tx } +|\mu|^2\dlinn{ x }  
                   +\dlinn{ \frac{TS}{\gl }x } -\frac{\mu}{\gl}\inn{ [S,T] x, x }\\
    &   =  \dlinn{ Sx }+\dlinn{ Tx }+  |\mu|^2\dlinn{ x } + \dlinn{ \frac{TS}{\gl }x } \\
    & \label{eq.RA.5}  \qquad + (1-\frac{\mu}{\gl})\inn{ [S,T]x,x } \\
    & \label{eq.RA.6}  \qquad+ \inn{ Tx,\frac{[S,T]}{\gl}x }+\inn{ \frac{[S,T]}{\gl}x,Tx }.
\end{align}

The expressions in the lines \Eqref{eq.RA.5} and \Eqref{eq.RA.6}
will be estimated separately from above and from below. Recall the estimates 
\Eqref{eq.DC.1} and \Eqref{eq.DC.2} as well as
\Eqref{eq.main.1} which will be used repeatedly.

\subsubsection*{Expression \Eqref{eq.RA.5}} We have
\begin{align*}
    \pm (1-\frac{\mu}{\gl})\inn{ [S,T]x,x } 
    &\le \frac 12 (1- \frac\mu\gl) \Bl \frac1{|\mu|} \dlinn{[S,T]x} + |\mu| \dlinn{x} \Br \\
    & \le \frac{C}{2|\mu|} \cdot \Bl \dlinn{Sx} + \dlinn{Tx} + \dlinn{x} \Br
          + \frac{|\mu|}2 \dlinn{x}.
\end{align*}

\subsubsection*{Expression \Eqref{eq.RA.6}}
\begin{align*}
   \pm \Bl  \inn{ Tx,\frac{[S,T]}{\gl}x }+\inn{ \frac{[S,T]}{\gl}x,Tx } \Br 
     &\le \frac1{|\gl|} \dlinn{Tx} + |\gl| \dlinn{\frac{[S,T]}{\gl} x } \\
   & = \frac1{|\gl|} \Bl \dlinn{Tx} + \dlinn{[S,T]x}\Br \\
   & \le \frac1{|\gl|} \bl C+1 \br \dlinn{Tx} + \frac{C}{|\gl|} \dlinn{Sx}
          + \frac{C}{|\gl|} \dlinn{x}.
\end{align*}
    
With these estimates we obtain
\begin{align*}
  \dlinn{(A_\gl + \mu)x} &\ge 
   \bl 1- \frac{C}{2|\mu|} - \frac{C}{|\gl|}\br \cdot \dlinn{Sx}
       + \bl 1- \frac{C}{2|\mu|} - \frac{C+1}{|\gl|}\br \cdot \dlinn{Tx} \\
    &\qquad + \bl |\mu|^2 - \frac{C}{2|\mu|} - \frac{C}{|\gl|} -
              \frac{|\mu|}2\br \cdot \dlinn{x}  + \dlinn{\frac{TS}{\gl} x} \\
\intertext{and since $|\mu|<|\gl|, 0 < \frac\mu\gl < 1$ }
    \ldots & \ge   \bl 1 - \frac{3/2 C+1}{|\mu|} \br  \bl \dlinn{Sx} + \dlinn{Tx}\br \\
     &\qquad  + \bl |\mu|^2 - \frac{3/2C}{|\mu|} - \frac{|\mu|}2\br \cdot \dlinn{x} 
            + \dlinn{\frac{TS}{\gl} x}.
\end{align*}
Choosing $\mu_0$ large enough we obtain for $|\mu|\ge \mu_0, |\gl|>|\mu|, \frac\mu\gl>0$
\[
  \ldots \ge \frac12 \dlinn{Sx} + \frac12 \dlinn{Tx} + \frac{|\mu|^2}{2} \dlinn{x}
            + \dlinn{\frac{TS}{\gl} x}.
\]
From this we infer for $x\in\sF(S,T)$ the inequality
$ \|(A_\gl +\mu )\ii x\|^2\le \frac{2}{|\mu|^2} \|x\|^2$, resp.
the claimed $\|(A_\gl +\mu )\ii \|\le \frac{\sqrt 2}{|\mu|}$.

Replacing $x$ by $(A_\gl+\mu)\ii x$ we find
$\| S(A_\gl +\mu )\ii \|\le \sqrt 2$,  $\| T(A_\gl +\mu )\ii \|\le \sqrt 2$, 
and $\|\frac{TS}{\gl}(A_\gl +\mu )\ii \| \le 1$.

Applying the form estimate \Eqref{eq.main.1} and taking norms yields
\begin{multline*}
     \| [S,T](A_\gl+\mu)\ii x\|^2\\
           \le C \cdot  \Bl \|(A_\gl+\mu)\ii x\|^2 
                + \|S (A_\gl+\mu)\ii x\|^2 + \|T (A_\gl +\mu )\ii \|^2 \Br \le C,
\end{multline*}
and the proof is complete.
\end{proof}

\subsection{Self-adjointness and regularity}

\begin{lemma}\label{p.RA.3} 
For $\mu_0>0$ large enough we have for $\gl,\mu\in i\R, |\gl|>|\mu|\ge\mu_0, \mu/\gl>0$
\begin{align*}
    (A_\gl +\mu )\ii \bl \domS \br
           &\subset \Sl \Tl \bl \domS \br\\ 
           &\subset \Dom{S^2} \cap \Dom{STS} \cap \Dom{ST}\cap \Dom{TS}. 
\end{align*}
\end{lemma}

\begin{proof} 
With $B_\gl  = A_\gl + \gl = \frac 1\gl (T+\gl)(S+\gl)$ as in \Eqref{eq.RA.1}
we have, \cf the proof of Lemma \ref{p.RA.1},
\begin{align}
\label{reseq}    (A_\gl +\mu )\ii &= B_\gl \ii - (\mu-\gl) B_\gl \ii (A_\gl +\mu )\ii \\
\nonumber       &= B_\gl \ii (1-(\mu-\gl)(A_\gl +\mu )\ii ).
\end{align}
The operator $1-(\mu-\gl)(A_\gl +\mu )\ii $ preserves the domain of $S$,
since
\[
   \ran (A_\gl +\mu)\ii =\sF (S,T)\subset\domS,
\] 
proving the first set inclusion. The second one follows from the identity
\[
   S (S+\gl)\ii =1-\gl(S+\gl )\ii ,
\]
the fact that $(T+\gl )\ii $ preserves $\domS$ and the fact that
\[
   \ran (T+\gl )\ii (S+\gl)\ii 
       =\ran  (S+\gl )\ii (T+\gl)\ii \subset \Dom{ST}\cap \Dom{TS}.
\]
This proves the lemma.
\end{proof}

\begin{lemma}\label{anticomm} Under the same assumptions
as in Lemma \plref{p.RA.3} we have for $x\in\domS$
the equality
\begin{multline*}
    S(S+ T+\frac{TS}{\gl } +\mu )\ii x+(S-T-\frac{ST}{\gl }-\mu )\ii Sx 
         \\ =(S-T-\frac{ST}{\gl }-\mu )\ii [S,T](S+T+\frac{TS}{\gl } +\mu )\ii x,
\end{multline*}
and each of these elements is in $\domT$.
\end{lemma}

\begin{proof}
By the Lemma \ref{p.RA.3} we may write, on $\domS$,
\begin{align*}
    S\bl S+T&+\frac{TS}{\gl } +\mu \br\ii + \bl T-S-\frac{ST}{\gl }-\mu \br\ii S\\ 
       &=\bl T-S-\frac{ST}{\gl }-\mu  \br\ii\cdot \\
         &\qquad \cdot  \Bl S \bl S+T+\frac{TS}{\gl } +\mu \br 
                  +  \bl T-S-\frac{ST}{\gl }-\mu  \br S \Br
                  \cdot   \bl S+T+\frac{TS}{\gl } +\mu  \br\ii\\
       &=\bl T-S-\frac{ST}{\gl }-\mu  \br\ii [S,T] \bl S+T+\frac{TS}{\gl } +\mu )\ii,
\end{align*}
as claimed. The second summand on the left hand side maps into $\domT$, as
does the right hand side, and thus the remaining term maps into $\domT$ as well.
\end{proof}

\begin{theorem}\label{p.RA.5} The operator $S+T$ is self-adjoint and regular on
$\domS\cap\domT$. More precisely, for $\mu\in i\R$ large enough the
net $\Bl (S+T+\mu )(A_\gl +\mu )\ii \Br_\gl , \gl\in i\R, |\gl|>|\mu|,
\mu/\gl>0,$ converges strongly to $1$ on $E$ as $|\gl|\to\infty$. The
module 
\[
  \sF (S,T)=\ran (S+\gl )\ii (T+\gl )\ii ,
\] 
is a core for $S+T$. The net $\Bl (S+T+\frac{TS}{\gl } +\mu )\ii \Br_\gl ,
\gl\in i\R, |\gl|>|\mu|, \mu/\gl>0,$
is norm convergent to the resolvent $(S+T+\mu )\ii $ of $S+T$ as
$|\gl|\to\infty$.
\end{theorem}
\begin{remark} In view of Prop. \ref{p.DC.1} the fact that $\sF(S,T)$
is a core for $S+T$ implies that for $x\in\domS\cap\domT$ there exists
a sequence $(x_n)\subset\sF(S,T)$ such that $x_n\to x, Sx_n\to Sx$,
and $Tx_n\to Tx$.
\end{remark}

\begin{proof}
By definition we have
\[
   \sF(S,T)\subset\domS\cap\domT,
\]
and by Proposition \ref{p.DC.1}, the domain of the closure of $S+T$,
restricted to $\sF(S,T)$, is contained in $\domS\cap\domT$. We will show that
\[
   S+T+\mu:\sF(S,T)\to E,
\]
has dense range for $\mu\in i\R$ with $|\mu|$ sufficiently large. Since $S+T$
is symmetric and closed on $\domS\cap \domT\supset \sF(S,T)$, this proves
self-adjointness and regularity of $S+T$ on this domain,
\cf \cite[Lemma 9.7]{Lan:HCM}. Furthermore, that
\[
\bl S+T+\mu\br\bl \sF(S,T) \br 
\]
is dense in $E$ is equivalent to the fact that $\sF(S,T)$ is a core for
$S+T$.

By Lemmas \ref{p.RA.1} and \ref{p.RA.2}, for $|\mu|\ge |\mu_0|$
large enough, the bounded adjointable operators
$(A_\gl +\mu)\ii$ map $E$ into $\sF(S,T)$. Furthermore,
the net 
\[
  \Bl (S+T+\mu )(A_\gl +\mu )\ii \Br_{|\gl|>|\mu|, \mu/\gl>0}
\]
is uniformly bounded in norm and 
\begin{align*}
 (S+T+\mu )(A_\gl +\mu )\ii =1-\frac{TS}{\gl } (A_\gl +\mu )\ii ,
\end{align*}
so it suffices to show that $\frac{TS}{\gl } (A_\gl +\mu )\ii$ converges to 
$0$ strongly on $\domS$. 
There, by virtue of Lemma \ref{anticomm}, we can write
\begin{multline*}
    \frac{TS}{\gl } (A_\gl +\mu )\ii x
         = \frac{-T}{\gl }(S-T-\frac{ST}{\gl }-\mu )\ii Sx\\ 
         +\frac{T}{\gl }(S-T-\frac{ST}{\gl }-\mu )\ii [S,T](S+T+\frac{TS}{\gl } +\mu )\ii x.
\end{multline*}
Using the estimates \eqref{eq.RA.2} and \eqref{eq.RA.3}, we see that both 
summands are $O(|\gl|\ii )$ in norm, and thus converge to $0$. We conclude
indeed that the operator $S+T+\mu:\domS\cap \domT\to E$ has dense range and
it is therefore bijective as explained above.
It follows that the resolvents $(S+T+\mu )\ii $ exist and that the submodule
\[
    \sF(S,T)=\ran (S+\gl )\ii (T+\gl )\ii \subset \domS\cap \domT,
\]
is a core for $S+T$. Now $S(S+T+\mu)\ii $ is a bounded adjointable operator so
\begin{align*}
    \|(A_\gl +\mu)\ii -(S+T+\mu)\ii  \|
          &= \| (A_\gl +\mu)\ii \frac{TS}{\gl}(S+T+\mu)\ii \|\\ 
          &\leq \frac{C}{|\gl|}\|S(S+T+\mu)\ii \|\to 0,
\end{align*} 
as $|\gl|\to \infty$, which completes the proof. 
\end{proof}

\section{Applications}
\label{s.appl}

We present two applications of the main Theorem \plref{p.main.7}.

\subsection{A Dore-Venni type theorem for $S^2+T^2$}
\label{ss.DVT}

\begin{theorem}\label{p.appl.1} Let $S,T$ be weakly anticommuting operators in the
Hilbert-$B$-module $E$. Then $S^2+T^2$ is self-adjoint and regular
on $\Dom{S^2}\cap \Dom{T^2}$. The latter equals
$\NDom\bl (S+ T)^2 \br$ and is a subset of $\sF(S,T)$.
\end{theorem}
\begin{remark} Ignoring domain questions for the moment one has
\[
  (S+T)^2  = S^2+T^2 +  [S,T] .
\]
A posteriori it indeed turns out that $[S,T]$ is relatively
bounded w.r.t. $(S+ T)^2$ with arbitrarily small
relative bound.  Hence once the domain claims are verified the Theorem is a
consequence of the Hilbert-$C^*$-module version of the Kato-Rellich Theorem
\cite[Theorem 4.5]{KaaLes:LGP}.
\end{remark}
\begin{proof}
By Theorem \plref{p.main.7} and by \Eqref{eq.main.1} the 
closure of the operator $[S,T]$ has $\domS\cap\domT$ in its
domain and $\ovl{[S,T]}\big|_{\domS\cap\domT}$ 
is symmetric. Furthermore, 
\[\domS\cap\domT\supset \bl\bl \Dom{S^2}\cap\Dom{T^2}\br \cup \LDom{(S+T)^2} \br.\]
In view of Theorem \ref{p.main.7} (4) and Theorem \ref{p.RA.5}
for $x\in\domS\cap\domT$ we choose an approximating
sequence $(x_n)\subset \sF(S,T)$ with $x_n\to x, Sx_n\to Sx$, and
$Tx_n\to Tx$. By \Eqref{eq.main.1} then also $([S,T]x_n)$ converges
to $\ovl{[S,T]}x$ and hence
\begin{align}
\label{eq.appl.1}
    \dinn{ \ovl{[S,T]} x } 
                    & \le C_0 \dinn{x} + C_1 \dinn{Sx} + C_2\dinn{Tx}\\
\nonumber    & \le C \bl \dinn{x} + \dinn{(S+T)x}\br.
\end{align}
Taking norms we find
\[
    \| \ovl{[S,T]}x \| \le C \bl \|x\| + \|(S+T)x\| \br.
\]
Thus\footnote{If $A$ is a self-adjoint and regular operator
in a Hilbert-$C^*$-module it follows from the Spectral Theorem
that for any $\eps>0$ there exists a $C_\eps$ such that 
for $x\in\Dom{A^2}$ one has
$\| Ax \| \le C_\eps\cdot \|x\| + \eps\cdot \| A^2 x\| $.
} 
for any $\eps>0$ there exists a $C_\eps>0$ such that
for $x\in \LDom{(S+T)^2}\subset \domS\cap\domT$ one
has 
\[
    \| \ovl{[S,T]}x \| \le C_\eps\cdot \|x\| +\eps\cdot \|(S+T)^2x\|.
\]
Thus on $\LDom{(S+T)^2}$ the symmetric operator $\ovl{[S,T]}$
is $(S+T)^2$ bounded with arbitrarily small bound.
The Kato-Rellich Theorem for regular operators \cite[Theorem 4.5]{KaaLes:LGP}
now implies that
$S^2+T^2=(S+T)^2-[S,T]$
is self-adjoint and regular on $\LDom{(S+ T)^2}$.

Furthermore, this operator is obviously symmetric
on the submodule
\[
  \Dom{S^2}\cap\Dom{T^2}.
\]
Therefore, if we can show that 
\[
 \LDom{(S+T)^2}\subset \Dom{S^2}\cap \Dom{T^2}\cap\sF(S,T)
\]
then we are done because a self-adjoint
and regular operator does not have proper symmetric extensions. 

To this end we use the Clifford matrices of Section \ref{ss.Clifford}.
Then 
\begin{equation}\label{eq.appl.2}
\haS\sigma_3 =\begin{pmatrix} S & 0 \\ 0 & -S\end{pmatrix}, \quad
\haT     =\begin{pmatrix} T & 0 \\ 0 & T\end{pmatrix}.
\end{equation}
We have the relations
\begin{equation}\label{eq.appl.3}
\begin{split}\hat S \sigma_{3} \cdot \hat T \pm \hat T \cdot \hat S \sigma_{3}
     &= (\hat S\hat T\pm\hat T\hat S)\cdot  \sigma_{3} , \\
    (\hat S\cdot \sigma_{3} \pm \hat T)\cdot \hat S \gs_{1} 
       + \hat S\gs_{1} \cdot(\hat S\cdot\sigma_{3}\pm\hat T)  
       & =\pm (\hat S\hat T + \hat T\hat S) \cdot \gs_{3} ,\end{split}
\end{equation}
so the pair $(\haS\sigma_3,\haT)$ is also weakly anticommuting.
Furthermore, 
\begin{equation}\label{eq.appl.4}
\haS\sigma_3 + \haT =\begin{pmatrix} S+T & 0 \\ 0 & T-S\end{pmatrix},
\end{equation}
and the pair $(\haS\sigma_3+\haT,\haS\sigma_1)$ is weakly anticommuting
as well, \cf the last equation in \Eqref{eq.appl.3}. Namely, $\sF(\haS\sigma_1,\haS\sigma_3+\haT)$ by definition equals the set
\begin{align*}\bigsetdef{x\in\Dom{\haS}\cap\Dom{\haT}}{ \haS
x\in\Dom{\haS}\cap\Dom{\haT}, \haT x\in\Dom{\haS}}, 
\end{align*}
and thus
\begin{align*}
  (\haS\sigma_1+\gl)\ii \bl\Dom{\haS}\cap\Dom{\haT} \br
   &\subset \sF(\haS\sigma_1,\haS\sigma_3+\haT).
\end{align*}
By Theorem \ref{p.main.7} \reftwo\ we thus also have the
inclusion
\[
  (\haS\sigma_3+\haT+\gl)\ii \bl \Dom{\haS}  \br \subset \sF(\haS\sigma_1,
\haS\sigma_3+\haT).
\]
In view of the matrix representation \Eqref{eq.appl.4}
this implies
\begin{multline}\label{eq.appl.5}
(S+ T+\gl)\ii \bl  \domS  \br \subset 
    \bigsetdef{x\in\domS\cap\domT}{ Sx\in\domS\cap\domT, Tx\in\domS}.
\end{multline}

Now we are ready to prove the inclusion 
$\LDom{(S+T)^2}\subset \Dom{S^2}\cap\Dom{T^2}$.
Namely, let $x\in\LDom{(S+T)^2}$.  Then $(S+T+\gl)x\in \domS\cap\domT$ and
hence by \Eqref{eq.appl.5} 
$x\in\domS\cap\domT, Sx\in\domS\cap\domT, Tx\in\domS$ and in particular $x\in\sF(S,T)$. 
A posteriori $Tx = (S+T)x-Sx\in\domT$ as well. Thus we have shown that
$x\in\Dom{S^2}\cap\Dom{T^2}\cap \sF(S,T)$.
\end{proof}
\begin{cor} Let $S,T$ be weakly commuting operators in the
Hilbert-$B$-module $E$. Then $S^2+T^2$ is self-adjoint and regular
on $\Dom{S^2}\cap \Dom{T^2}$ and the latter equals
\[
   \NDom\bl (S+ iT)^{*}(S+iT) \br=\NDom\bl (S-iT)^{*}(S-iT)\br.
\]
\end{cor}
\begin{proof} \Cf Section \ref{ss.Clifford}, the operators
\[
   \hat S\sigma_1=\begin{pmatrix}0 & S \\ S & 0\end{pmatrix}, 
   \quad \hat T\sigma_2=\begin{pmatrix}0 & iT \\ -iT & 0 \end{pmatrix},
\]
are weakly anticommuting. The previous Theorem gives that
\[
  (\hat S\sigma_1)^2+(\hat T\sigma_2)^2=\hat S^2+\hat T^2,
\]
is self-adjoint and regular on 
\[
   \Dom{\hat S^2}\cap \Dom{\hat T^2}= \Bl \Dom{ S^2 }\cap \Dom{ T^2} \Br^{\oplus 2},
\]
from which we conclude that $S^2+T^2$ is self-adjoint and regular on
$\Dom{ S^2 }\cap\Dom{ T^2 }$. The domain equality
\[
   \LDom{ \bl\hat S\sigma_1+\hat T\sigma_2\br^2 }=\LDom{(S-iT)^{*}(S-iT)}\oplus \LDom{(S+iT)^{*}(S+iT)},
\]
proves the remaining claim.
\end{proof}

\subsection{Iteration}

In our early discussions on this project the following result appeared as a
problem and seemed to be a major step in the second author's program to find a
suitable category of unbounded $KK$-cycles. Therefore it was one of the
driving forces to develop the sum theory of this paper.  Although it became
clear that Theorem \ref{p.appl.4} does not resolve all remaining issues we think it is
useful to record here for future reference. 

\begin{theorem}\label{p.appl.4}
Let $S_j, j=1,2,3$, be self-adjoint and regular operators
in the Hilbert-$B$-module $E$. Assume that $(S_1,S_2), (S_2, S_3),
(S_1,S_3)$ are weakly anticommuting pairs. Then $S_1+S_2$ and $S_3$ are weakly
anticommuting and hence $S_1+S_2+S_3$ is self-adjoint and regular on 
$\Dom{S_1}\cap \Dom{S_2}\cap \Dom{S_3}$.
\end{theorem}

Let us quickly explain why this is 
remarkable.\footnote{It is more remarkable when one compares with the original
definition of weak anticommutativity, \cf Sec. \ref{ss.CPNWA}.}
According to the definition one needs to verify that 
$(S_1+S_2+\gl)\ii\bl \Dom{S_3}\br \subset \sF(S_1+S_2,S_3)$. 
A priori this is difficult since we have no concrete
information about the resolvent of $S_1+S_2$. However due to Theorem
\ref{p.main.7} the definition of weak anticommutativity 
is symmetric in the operators. Hence we know that also 
$(S_3+\gl)\ii \bl \Dom{S_j} \br \subset \sF(S_j,S_3)$
and, more importantly, that it suffices to show that
$(S_3+\gl)\ii \bl \Dom{S_1} \cap \Dom{S_2} \br \subset \sF(S_1+S_2,S_3)$.
The latter turns out to be straightforward.

\begin{proof}
By Theorem \plref{p.main.7} the operator $S_1+S_2$ is self-adjoint
and regular on $\Dom{S_1}\cap\Dom{S_2}$. By assumption we have
\[
\begin{split}
     (S_3+\gl)\ii \bl \Dom{S_1} \br  &\subset \sF(S_1,S_3),\\
     (S_3+\gl)\ii \bl \Dom{S_2} \br  &\subset \sF(S_2,S_3),
\end{split}
\]
thus
\[
\begin{split}
     (S_3+&\gl)\ii \bl \Dom{S_1+S_2} \br  \subset \sF(S_1,S_3) \cap \sF(S_2,S_3)\\
        &= \bigl\{x\in\Dom{S_1}\cap\Dom{S_2}\cap\Dom{S_3} \,\bigm|\,\\
        & \qquad     S_1x, S_2x\in\Dom{S_3}, S_3x\in\Dom{S_1}\cap\Dom{S_2}\bigr\}  \\
     & \subset \sF(S_1+S_2,S_3).
\end{split}
\]
By the Remark \ref{p.main.3} \refthree\ and by Theorem
\plref{p.main.7} this implies
\[
     \sF(S_1+S_2,S_3) = (S_3+\gl)\ii \bl   \Dom{S_1+S_2}\br
      =\sF(S_1,S_3) \cap \sF(S_2,S_3).
\]
On $\sF(S_1+S_2,S_3)$ we thus have, using the abbreviation
\Eqref{eq.convenient-abbrev},
\[
\begin{split}
 \dlinn{ [S_1+S_2, S_3] x} & = 
           \dlinn{ [S_1,S_3]x} + \dlinn{ [S_2,S_3]x } \\
           & \quad \inn{ [S_1,S_3]x, [S_2,S_3]x } + \inn{ [S_2,S_3]x,
[S_1,S_3]x } \\
       & \le 2 \bl  \dlinn{ [S_1,S_3]x } + \dlinn{ [S_2,S_3]x } \br \\
       & \le  C'_0 \cdot \inn{x,x}  + C'_1 \cdot \inn{ S_1 x, S_1 x}
                     + C'_2 \inn{S_2 x, S_2 x}  + C'_3\cdot\inn{S_3x, S_3 x}  \\
       & \le  C_0 \cdot \inn{x,x} + C_1 \cdot \inn{(S_1+S_2)x, (S_1+S_2)x} 
          + C_2 \cdot \inn{S_3 x, S_3 x},
\end{split}
\]
where we used \Eqref{eq.DC.1} and the form estimate for the weakly
anticommuting pairs $(S_1,S_3), (S_2,S_3)$. Hence $S_1+S_2,S_3$ is a weakly
anticommuting pair of operators.
\end{proof}

\section{Comparison to Banach space results for sums of operators}
\label{s.CBS}

There exists quite some literature on the problem of closedness
and regularity of sums of sectorial operators in Banach spaces
\cite{DPG1975,DV1987,LT1987,Fuh1993,MP1997,KW2001,PS2007,Roi2016}.
One might wonder whether and how our results compare to these.
After all a Hilbert-$C^*$-module is a Banach space and our operators
$S,T, S^2, T^2$ are \emph{sectorial} operators. It is the purpose
of this section to put this into context. 

Let $X$ be a Banach space and let $A$ be a densely defined
operator in $X$. One should think of $A$ as being $S^2$ or $T^2$ above.
The operator $A$ is called \emph{sectorial} 
in the open sector
$\Sigma_\theta:=\bigsetdef{ z\in\C}{ z\not=0, |\arg z| < \theta}$
if $\ker A=\{0\}$, $\ran A$ is dense, 
$A+\gl$ is invertible for $\gl \in \Sigma_\theta$
and 
\[
   M_\theta:= \sup_{\gl \in \Sigma_\theta} |\gl| \cdot  \| (A + \gl)\ii \| < \infty.
\]
The \emph{spectral angle} of $A$ is defined by
\[
    \phi_A:= \inf\bigsetdef{\phi>0}{%
            A \text{ is sectorial in } \Sigma_{\pi-\phi} }.
\]

Clearly, if $S$ is a self-adjoint and regular operator in the
Hilbert-$C^*$-module $E$ then $S^2$ is sectorial with spectral angle $0$,
while for any $\delta>0$ the operator $iS+\delta$ is sectorial with
spectral angle $\pi/2$. 

Now given two sectorial operators $A,B$ in a Banach space, the analogue
of the regularity problem in Hilbert-$C^*$-modules splits into the
subproblems to decide whether $A+B$ is closed on $\Dom{A}\cap\Dom{B}$
and whether $A+B+\gl$ has dense range for $\gl>0$ large (resp.
$\ovl{A+B}$ is sectorial). This should be compared to Theorem \ref{p.main.7}
\mainclosedness, \mainstrong\ and to Theorem \ref{p.appl.1}.

One of the seminal results on this is the following:

\begin{theorem}[\textup{Da Prato and Grisvard \cite{DPG1975},
 \cf also \cite[Sec.~3]{PS2007}}]
\label{p.CBSR.1}
Let $A,B$ be sectorial operators in a Banach space $X$ with spectral
angles $\phi_A,\phi_B$. Assume that $\psi_A>\phi_A, \psi_B>\phi_B, 
\psi_A+\psi_B<\pi$, 
$(A+\gl)\ii \bl \Dom{B} \br \subset \Dom{B}$, and
\begin{equation}\label{eq.CBSR.1}
\bigl\| \bl B (A+\gl)\ii - (A+\gl)\ii B\br (\mu+ B)\ii \bigr\|\\
    \le \frac{c}{(1+|\gl|)^\ga |\mu|^\gb},
\end{equation}
for $\gl\in \Sigma_{\pi-\psi_A}, \mu\in\Sigma_{\pi-\psi_B}$
and fixed numbers $ \ga,\gb>0, \gb<1, \ga+\gb>1$.

Then $\ovl{A+B}$ is sectorial with spectral angle
$\le \max(\psi_A, \psi_B)$.
\end{theorem}
In \cite{LT1987} \Eqref{eq.CBSR.1} was replaced by a more flexible
but also more involved estimate, \cf \cite[Sec.~3]{PS2007}.

Proving closedness of $A+B$ on $\Dom{A}\cap\Dom{B}$ requires
additional assumptions on the Banach space (class $\sH\sT$) and
that $A,B$ admit bounded imaginary powers. The seminal
result is that of Dore and Venni \cite{DV1987} for resolvent
commuting operators. This was later improved by
Monnieux and Pr\"uss \cite{MP1997} for non-commuting operators
satisfying the above mentioned Labbas-Terreni
\cite{LT1987} commutator condition, and by Pr\"uss and Simonett
\cite{PS2007} for pairs satisfying the Da Prato-Grisvard
commutator condition \Eqref{eq.CBSR.1}.

For our operators in a Hilbert-$C^*$-module $E$ bounded
imaginary powers are a non-issue. That is, for a self-adjoint
and regular operator $S$ in $E$ it follows trivially from
the continuous functional calculus that $S$ has bounded imaginary powers,
i.e. that $\bigsetdef{ S^{it} }{ -1 \le t \le 1 }$ is
bounded in $\sL(E)$.

The class $\sH\sT$ condition seems to be ``orthogonal'' to
Hilbert-$C^*$-module theory: recall that a Banach space
$X$ is called of class $\sH\sT$ if the Hilbert transform
\[
      Hf(t) := \lim_{\eps\searrow 0} \int_{|s|\ge \eps}
         f(t-s) \frac{ds}{s}, \quad f\in\cinfz{\R, X},
\]
extends by continuity to a bounded linear operator $L^2(\R,X)\to L^2(\R,X)$.
Here, $L^2(\R,X)$ is the completion of $\cinfz{\R,X}$ with
respect to the norm
\[
   \| f\| := \left(\int_\R \| f(x) \|_X^2 dx \right)^{1/2}.
\]
We asked experts on the aforementioned Banach space theory
but the following problem could not be clarified:

\begin{problem} Let $E$ be a Hilbert-$B$-module.
Decide whether $E$ is of class $\sH\sT$ or not.
\end{problem}

The problem here is that in general $L^2(\R,E)$ 
is not a Hilbert-$B$-module, but is
strictly smaller than the external tensor product
$L^2(\R)\hat\otimes_\C E$. The latter is the completion
of $\cinfz{\R,E}$ with respect to the inner product
\[
    \inn{f,g} = \int_\R \inn{f(x),g(x)}_E dx \in B
\]
resp. the induced norm
\[
    \| f\| = \Bigl\|  \int_\R \inn{f(x),g(x)}_E dx \Bigr\|_B^{1/2}.
\]
Clearly, since the Hilbert transform $H_0: L^2(\R)\to L^2(\R)$ is bounded 
the Hilbert transform on $L^2(\R)\hat\otimes_\C E$ is nothing but
$H_0\hat\otimes\id$, which is bounded as well. This leads to
 
\begin{problem} Is it true that for Hilbert-$C^*$-modules
the proof of the above mentioned Dore-Venni type results
go through by exploiting instead of the $\sH\sT$ condition
the (obviously true) condition of boundedness of the Hilbert transform on
$L^2(\R)\hat\otimes_\C E$.
\end{problem}

Finally, we outline an alternative approach to the main Theorem 
\ref{p.main.7} using the method due to Da Prato-Grisvard. While
their results have been generalized and refined, all subsequent
publications essentially employ the basic pattern which can already
be found in \cite{DPG1975}. Namely, given sectorial operators
$A, B$ then view $A+B+\gl$ as a (operator valued) function of
$B$. Then the resolvent $(A+B+\gl)\ii$ should be given by the
Dunford integral
\begin{equation} \label{eq.CBSR.2}
    P_\gl := \frac{1}{2\pi i} \int_\Gamma (z+\gl+A)\ii \cdot (z-B)\ii dz,
\end{equation}
where $\Gamma$ is a contour of the form
$(\infty,r) e^{i\theta} \cup r e^{i[\theta,\pi-\theta]}\cup (r,\infty) e^{-i
\theta}$ with $\psi_B<\theta<\min(\psi,\pi-\psi_A)$.

If $A$ and $B$ are resolvent commuting, then $P_\gl$ equals the
resolvent $(A+B+\gl)\ii$. In all other cases $P_\gl$ is only
an approximation to the resolvent and the main part of the work
is to formulate commutator conditions on $A$ and $B$ ensuring
that $P_\gl$ maps a sufficiently large space into $\Dom{A}\cap\Dom{B}$.

Turning to a pair of weakly anticommuting operators $S,T$ we cannot
apply the pattern outlined above since our commutator conditions
Def. \ref{p.main.1} concern commutators of $S$ and $T$ but not
of $S^2, T^2$. Therefore, it is unrealistic to prove commutator
estimates on $S^2, T^2$ \`a la Da Prato-Grisvard resp. Labbas-Terreni.

However, we can slightly modify \Eqref{eq.CBSR.2} to obtain
a resolvent approximation of $S+T+i\gl$ instead of $S^2+T^2+\gl$.\footnote{In
this section $\gl,\mu$ denote real parameters.}

Namely, from the estimate \Eqref{eq.main.4} one infers
\[
   \| [S,T] (S+i\gl)\ii (T+i \mu)\ii \| \le C \bl \frac{1}{|\gl|} +
\frac{1}{|\mu|}\br
\]
hence, \cf Theorem \ref{p.CBSR.1}, 
\[
  \bigl \| [ T^2, (S^2+\gl)\ii ] (T^2+\mu) \ii \bigr\|
     \le  \frac{C}{|\gl|}\bl \frac{1}{\sqrt{|\gl|}}+ \frac{1}{\sqrt{|\mu|}}\br.
\]
However, we may not expect the stronger domain inclusion
\[
(S^2+\gl)\ii \bl \Dom{T^2} \br \subset \Dom{T^2}
\]
to hold.

Nevertheless, without further assumptions,
these estimates and the axioms of weak anticommutativity
allow to prove
\begin{theorem}\label{p.CBSR.4}  Let
\[
   P_\gl := \frac{1}{2\pi i} \int_\Gamma (z+\gl+S^2)\ii (S+T - i \gl) (z -
   T^2)\ii dz.
\]
Then for $y\in \domS\cap\domT$ and $\gl$ large
we have $P_\gl y\in\domS\cap\domT$ and
\[
   (S+T+i\gl) P_\gl y = (I+R_\gl) y, 
\]
with $\|R_\gl\|<1$, hence $ran (S+T+i\gl)$ dense.
\end{theorem}

So in principle the original idea of Da Prato-Grisvard together with
the axioms of weak anticommutativity lead to yet another proof of
the closedness and regularity statements in Theorem \ref{p.main.7}.
Further details are omitted and hence left to the reader.

\section{The Kasparov product of unbounded modules}
\label{s.KK}

In this section we describe how Theorem \ref{p.RA.5} can be applied in the
constructive approach to the Kasparov product. For background on unbounded
$KK$-theory we refer to 
\cite{BaaJul:TBK,Bla:KTO2Ed,Con1994,Kuc:KKP,Kuc2000,Mes2014,MR2016,KaaLes:SFU}.

\subsection{Weakly anticommuting operators and the Kasparov product}

Kasparov's $KK$-theory \cite{Kas:OKF} is a powerful tool in operator
$K$-theory \cite{Bla:KTO2Ed}.  It associates to a pair of separable
$C^{*}$-algebras $(A,B)$ an abelian group $KK_0(A,B)$.  The main feature of
$KK$-theory is the existence of an associative bilinear product
\begin{equation}
\label{eq.KK.1}
KK_0(A,B)\times KK_0(B,C)\to KK_0(A,C),
\end{equation}
defined for all separable $C^{*}$-algebras $A,B$ and $C$. 

A $\Z/2$-\emph{grading} on a Hilbert $C^{*}$-module $E$ is a self-adjoint
operator $\gamma\in\sL(E)$ such that $\gamma^2=1$.  An operator $F\in\sL(E)$
is \emph{even} if $F\gamma=\gamma F$ and \emph{odd} if $\gamma F=-F\gamma$.

Elements of the group $KK_0(A,B)$ are given by the following data:

\begin{dfn}\cite{Kas:OKF} Let $(A,B)$ be a pair of separable $C^{*}$-algebras.
A \emph{Kasparov module} for $(A,B)$ is a pair $(E,F)$ where 
\begin{enumerate}
\item $E$ is a $\Z/2$-graded Hilbert $C^{*}$-module over $B$ together with a
      $*$-homomoprhism $A\to \sL(E)$;
\item $F\in\sL(E)$ is an odd operator such that $a(1-F^2), a(F-F^{*})$ and
      $[F,a]$ are elements of $\sK(E)$.
\end{enumerate}
\end{dfn}
Here, $[\cdot,\cdot]$ denotes the \emph{graded} commutator which on
homogeneous elements $x,y$ of parity $\pl x, \pl y$ is defined by 
$[x,y] = xy - (-1)^{\pl x\cdot \pl y} yx$. That is $[a,b]=[a,b]_+$
if $a,b$ are both odd and $[a,b]=[a,b]_-$ if one of them is even.
For the commutators introduced in Section \ref{ss.WAO} we will therefore
always write $[\cdot,\cdot]_+$ resp.  $[\cdot,\cdot]_-$ to make the sign of
the second summand explicit.

For an $(A,B)$ Hilbert bimodule $E$ we will refer to the $C^{*}$-algebra 
\[
   \bigsetdef{K\in\sL(E)}{\forall a\in A \quad aK,Ka\in\sK(E)}
\]
as the $C^{*}$-algebra of $A$-\emph{locally compact operators on} $E$. For
self-adjoint elements $Q,R\in\sL(E)$ we say that $Q\leq R$ \emph{modulo}
$A$-\emph{locally compact operators} if there exists a locally compact
operator $K$ such that $Q\leq R+K$.

In \cite[Theorem A.5]{ConSka} Connes-Skandalis provided sufficient conditions
that determine the product \eqref{eq.KK.1}.
\begin{theorem}\label{thm.bddprd}
Let $(X,F_{X})$ and $(Y,F_{Y})$ be Kasparov modules for $(A,B)$ and $(B,C)$
respectively.  Suppose that $(X\otimes_{B}Y,F)$ is a $(A,C)$ Kasparov module
such that
\begin{enumerate}
\item for all $x\in X$ the operator $y\mapsto \gamma(x)\otimes F_{Y}y-F(x\otimes y)$
      is in $\sK(Y,X\otimes_{B}Y)$;
\item there is $0\leq \kappa<2$ such that for all $a\in A$ the operator inequality
\[
          a^{*}[F_{X}\otimes 1, F]a\geq -\kappa a^{*}a,
\]
holds modulo $A$-compact operators.
\end{enumerate}
Then $(X\otimes_{B}Y, F)$ represents the Kasparov product of $(X,F_{X})$ and
$(Y, F_{Y})$.
\end{theorem}
Notice that condition (ii) is weaker than what is stated in 
\cite[Theorem A.5]{ConSka} (see \cite[Definition 4]{Kuc:KKP} and 
\cite[Definition 18.4.1]{Bla:KTO2Ed}). This weakening will be of vital
importance for our main theorem.

In order to describe the external product in $KK$-theory in a constructive
way, Baaj-Julg introduced the following refinement of Kasparov modules.
\begin{dfn}[\cite{BaaJul:TBK}] Let $(A,B)$ be a pair of separable
$C^{*}$-algebras. An \emph{unbounded Kasparov module} for $(A,B)$ is a triple
$(\sA,E,D)$ where 
\begin{enumerate}
\item $E$ is a $\Z/2$-graded Hilbert $C^{*}$-module over $B$ together with a
      $*$-homomoprhism $A\to \sL(E)$;
\item $D:\sD(D)\to E$ is a self-adjoint regular operator such that 
      $a(D\pm i)\ii\in \sK(E)$ for all $a\in A$;
\item $\sA\subset A$ is a norm dense $*$-subalgebra such that 
      $a:\Dom{D}\to\Dom{D}$ and $[D,a]$ extends to an element in $\sL(E)$
      for all $a\in\sA$.
\end{enumerate}
\end{dfn}

A continuous function $\chi:\R\to [-1,1]$ is called a
\emph{normalizing function} if 
\[
    \chi(-x)=-\chi(x)\quad\textnormal{and}\quad \lim_{x\to\pm\infty}\chi(x)=\pm 1.
\]
If $(E,D)$ is an unbounded Kasparov module then $(E,\chi(D))$ is a Kasparov
module (\cite{BaaJul:TBK}) whose class does not depend on the choice of
$\chi$.  Notice that the difference of any two normalizing functions
$\chi_{1},\chi_{2}$ is an element of $C_0(\R)$, which is generated
by $(x\pm i)\ii $. Since $(D\pm i)\ii $ are locally compact, so is
$\chi_{1}(D)-\chi_{2}(D)$ and the two functions give homotopic Kasparov
modules, \cf \cite[Sec. 10.6]{HigRoe:AKH}.

\begin{theorem}  \label{thm.unbddprd}
Let $(\sA, X,S_{X})$ and $(\sB, Y,T_{Y})$ be unbounded Kasparov modules for
$(A,B)$ and $(B,C)$ respectively and let $E:=X\otimes_{B}Y$ and
$S:=S_{X}\otimes 1$. Suppose that $T:\domT\to E$ is an odd self-adjoint
regular operator such that
\begin{enumerate}
\item there is a dense $\sB$-submodule $\sX\subset\domS\subset X$ for which
the algebraic tensor product $\mathcal{X}\otimes^{\textnormal{alg}}_{\sB}\sD(T_{Y})$
is a core for $T$ and for all homogenous elements $x\in\sX$ and all $y\in \Dom{T_{Y}}$ 
the operator
\begin{align*}
      y &\mapsto \gamma(x)\otimes T_{Y}y-T(x\otimes y)
\end{align*}
defines an element of $\sL(Y,E)$ ;
\item for all $a\in\sA$ we have $a:\domT\to\domT$ and $[T,a]\in\sL(E)$;
\item $(S,T)$ is a weakly anticommuting pair.
\end{enumerate}
Then $(\sA,E,S+T)$ is an unbounded Kasparov module that represents the
Kasparov product of $(X,S_{X})$ and $(Y,T_{Y})$.
\end{theorem}
\begin{proof}
The fact that $(E,S+T)$ is an unbounded Kasparov module follows quite easily:
the sum $S+T$ is self-adjoint and regular by condition (iii), and has
bounded commutators with $\sA$ by condition (ii). By condition (i) and
\cite[Lemma 4.3]{MR2016} we have $a(S+\gl)\ii (T+\mu)\ii \in\sK(E)$ for
all $a\in \sA$. In particular
\[
    aB_\gl \ii =\frac a\gl \Sl \Tl\in\sK(E),
\]
with $B_\gl $ as in Lemma \ref{p.RA.1}. By \Eqref{reseq} we have 
 \[
     (A_\gl +\mu)\ii =B_\gl \ii -(\mu-\gl)B_\gl \ii (A_\gl +\mu)\ii ,
 \]
and it follows that $a(A_\gl +\mu)\ii \in\sK(E)$. By Theorem \ref{p.RA.5} 
\[
     a(S+T+\mu )\ii =\lim_{\gl\to \infty}a(A_\gl + \mu)\ii ,
 \]
is a norm limit and we conclude that $a(S+T+\mu )\ii \in\sK (E)$.
 
To show that $(E,S+T)$ represents the Kasparov product,
consider the normalizing functions
\[
   \chi(x):=\frac{2}{\pi}\arctan (x),\quad b(x):=x(1+x^2)^{-1/2}.
\]
We will prove that $\chi(D)$ satisfies the conditions of Theorem
\ref{thm.bddprd}.  By \cite[Proposition 14]{Kuc:KKP} the operators $b(D)$ and
$b(T_{Y})$ satisfy condition (i) of Theorem \ref{thm.bddprd}. Since
$\chi(D)-b(D)$ is $A$ locally compact on $E$ and $b(T_{Y})-\chi(T_{Y})$ is $B$
locally compact on $Y$, $\chi(D)$ and $\chi(T_{Y})$ satisfy condition (i) as
well.

The fact that  after a suitable homotopy, $\chi(D)$ and $\chi(S)$, satisfy
condition (ii) follows from Proposition \ref{localestimate} in Section
\ref{subsec.pos}.
\end{proof}
 
\begin{remark} 
Theorem \ref{thm.unbddprd} should be compared to \cite[Theorem 13]{Kuc:KKP}.
There, fewer assumptions are imposed on the form of the product operator, in
particular it need not arise as a sum.  The case where
\[
  \inn{[S,T]x,[S,T]x}  \leq C (\inn{ x,x }+ \inn{ Sx,Sx }),
\]
is covered by the latter result. This assumption was in place in
\cite{KaaLes:SFU, MR2016}.  However, as soon as there is a nontrivial relative
bound to $T$ as well, condition (iii) of \cite[Theorem 13]{Kuc:KKP} may not be
satisfied. An example of such a situation is given in \cite{BM2018}.
\end{remark}
 
\begin{remark}
The construction of operators $T$ satisfying hypotheses (i) and (ii) of
Theorem \ref{thm.unbddprd} is the subject of the of the papers
\cite{KaaLes:SFU, Mes2014, MR2016}. Indeed in \cite{MR2016} it was
shown that up to equivalence, such a $T$ can always be constructed. In
geometric situations, an operator $T$ with the required properties can often
be written down explicitly, see for example \cite{BMS2016, KvS2016} .
\end{remark}

\subsection{A form estimate for the absolute value of the sum}

We denote by $S$ and $T$ a weakly anti-commuting pair of operators on the
Hilbert $C^{*}$-module $E$, and by $D:=S+T$ their sum operator, which is
self-adjoint and regular. Our goal is to obtain a form estimate for the
anticommutator $[S,T]$ relative to the positive operator $|D|$ defined through
functional calculus.  As we wish to work on the domain of $D$ we consider the
extension $\ovl{[S,T]}$, as in the proof of Theorem \ref{p.appl.1}.

\begin{lemma}
For $0\leq \Re z\leq 1$ the operator
$P_{z}:=(1+|D|)^{-z}\ovl{[S,T]}(1+|D|)^{z-1}$ is bounded on $\Dom{S}\cap\Dom{T}$, extends to an adjointable operator 
and $\|P_{z}\|\leq \|P_0\|$.
\end{lemma}

\begin{proof} 
The operator $1+|D|:\Dom{D}\to E$ is boundedly invertible and by
\Eqref{eq.appl.1} the operator $\ovl{[S,T]}:\Dom D \to E$ is bounded,
when $\Dom D$ is equipped with the graph norm of $D$. Hence 
\[
    P_0:=\ovl{[S,T]}(1+|D|)\ii :E\to E
\]
is bounded and consequently the densely defined operator 
\[
 (1+|D|)\ii \ovl{[S,T]}:\sD(S)\cap\sD(T)\to E
\] 
is bounded as well and its closure $P_1$ equals the adjoint of $P_0$.
 
We now adapt the interpolation argument of \cite[Appendix A]{Les:USF} to the
case of Hilbert $C^{*}$-modules. For $\Re z>0$ the operators $(1+|D|)^{-z}$
preserve $\Dom{D}$. For $x,y\in \Dom{D}$ the function
\[
     f_{x,y}: z\mapsto \inn{ (1+|D|)^{-z}\ovl{[S,T]}(1+|D|)^{-1+z}x,y } = \inn{ P_{z}x,y}, 
 \]
is weakly holomorphic on the strip $0\leq \Re z\leq 1$. Since
\begin{align*}
   \dinn{ P_{z}x } &\leq \dinn{ x }+ \dinn{ D(1+|D|)^{-1+z}x }\\
        &\leq \dinn{ x } +\dinn{ Dx },
\end{align*} 
we infer that $\|P_{z}x\|\leq \|x\|+\|Dx\|$ and
\[
      \|f_{x,y}(z)\|\leq \|P_{z}x\|\|y\|\leq (\|x\|+\|Dx\|)\|y\|
\] 
so $f_{x,y}$ is a bounded function. Now let $\varphi:B\to\C$ be a
state. The function $z\mapsto \varphi\circ f_{x,y}(z)$ is bounded and
holomorphic in the strip $0\leq \Re z\leq 1$. By the Phragm{\'e}n-Lindel\"of
Theorem (aka Hadamard 3-line Theorem in this case) the function is bounded
by its suprema on the boundary $\Re z\in\{0,1\}$. For such $z$ it holds that
$\|P_{z}\|=\|P_0\|=\|P_1= P_0^*\|$. So we obtain that
\begin{align*}
    |\varphi(\inn{ P_{z}x,y })| 
        & \leq \sup_{\Re w\in \{0,1\}} |\varphi  (\inn{ P_w x,y } )|\\
        & \leq \sup_{\Re w\in \{0,1\}} \| \inn{ P_w x,y } \|  \leq \|P_0\| \| x \| \| y\|.
\end{align*}
Since this holds for all states $\varphi$ it follows that 
$\| \inn {P_z x,y } \| \leq \| P_0 \| \| x \| \| y\|$ 
and hence $\|P_z\| \le \|P_0\|$.
\end{proof}

By a \emph{rescaling} of a weakly anticommuting pair $(S,T)$ we mean a weakly
anticommuting pair of the form $(\gl S, \gl T)$ for some $\gl >0$.

\begin{prop} \label{estimates}
Let $(S,T)$ be a weakly anticommuting pair and $D=S+T$ their sum. For all
$x\in\domS\cap\domT$ the form estimate
\begin{align} \label{eq.KK.2}
\inn{ \ovl{[S,T]}x,x} \leq  C(\inn{ x,x} + \inn{ |D|x,x}),
\end{align}
holds true, with $C$ a constant independent of $x$. Consequently, for all
$\mu>0$ we have the operator estimate
\begin{align} \label{eq.KK.3}
\pm (1+\mu^2D^2)\ii [S,T](1+\mu^2D^2)\ii \leq C(1+|D|)(1+\mu^2D^2)^{-2}.
\end{align}
After a suitable rescaling of the pair $(S,T)$ we can achieve that
$C<\eps$ for any $\eps>0$.
\end{prop}

\begin{proof}
The operator $P_{1/2}$ is self-adjoint whence
\begin{align*}
\inn{ [S,T]x,x} &=\dinn{ P_{1/2}(1+|D|)^{1/2}x } \\
                &\leq \|P_0\| \dinn{(1+|D|)^{1/2}x}\\
             &=\|P_0\|(\dinn{ x} + \inn{ |D|x,x}),
\end{align*}
which proves the form estimate \eqref{eq.KK.2} with $C=\|P_0\|$.
The operator estimate \eqref{eq.KK.3} now follows in a
straightforward manner.  Replacing $S,T$ by $\gl S,\gl T$ for $0<\gl <1$ we
obtain
\[
    \inn{ [\gl S,\gl T]x,x }  = \gl^2 \inn{ [S,T]x,x} 
          \leq  \gl^2 C(\inn{ x,x} + \inn{ |D|x,x})
          \leq  \gl C  (\inn{ x,x} +\inn{ |\gl D|x,x}.
\]
Thus, by taking $\gl$ sufficiently small we may assume that $C$ is as small as
we like.
\end{proof}

\subsection{Proof of the positivity condition}\label{subsec.pos}

We use the integral representation of the function $\arctan=\tan\ii$
\[
   \arctan(x) = \int_0^x \frac 1{1+t^2} dt = \int_0^1 \frac x{1+\mu^2 x^2} d\mu.
\] 
For any self-adjoint regular operator $D$, the bounded adjointable operator
$\chi(D):=\frac{2}{\pi}\arctan (D)$ then has the representation
\[
\chi(D)=\frac{2}{\pi}\int_0^{1}D(1+\mu^2D^2)\ii d\mu,
\]
as a \emph{strongly} convergent integral (\cf \cite[Lemma 8]{Kuc:KKP}).

We now consider a weakly anticommuting pair of operators $(S,T)$ in a Hilbert
$C^{*}$-module $E$. Recall from Section \ref{ss.Clifford} that the
Clifford algebra $\Cltwo$ is represented unitarily on $E\oplus E$ and that the
$\Cltwo$ action commutes with $\hat S, \hat T$ (\eqref{Shat}, \eqref{That})
and that the action preserves their domains. Denote by $\omega:=\sigma_3= i
\sigma_1\sigma_2$ the volume element of $\Cltwo$ and we let $\hat D_\pm:=\hat
S\pm\go\hat T$ and $\hat D:=\hat D_+$, \cf \Eqref{eq.appl.2}--\eqref{eq.appl.4}.
Since $\go$ commutes with $\hat S,\hat T$
we have that the pair $(\hat S,\pm\go \hat T)$
is weakly anticommuting as well and that $\hat S\pm\go \hat T$ is self-adjoint and
regular with domain $\Dom{\hat S}\cap\Dom{\hat T}$. Recall also that
\[
\hat S+\go\hat T =\begin{pmatrix} S+T & 0 \\ 
                           0 & S-T \end{pmatrix}.
\]

So for the time being we may w.l.o.g. omit the hat
decorator and assume that $S,T$ are $\Cltwo$ invariant. 
         
\begin{lemma} For $\mu>0$ the operator
\begin{align}\nonumber 
    K_{\mu} &:=(1+\mu^2D_-^2)\ii -(1+\mu^2D_+^2)\ii \\
\label{eq.KK.4}  &=2 (1+\mu^2D_+^2)\ii \mu^2[S,T](1+\mu^2D_-^2)\ii \\
\label{eq.KK.5}  &=2 (1+\mu^2D_-^2)\ii \mu^2[S,T](1+\mu^2D_+^2)\ii 
\end{align}
is $A$ locally compact, as are the operators $D_{\pm}K_{\mu}$. Moreover
\[
    \sup\{\|K_{\mu}\|,\|D_{\pm}K_{\mu}\|:\mu\in (0,\infty)\}<\infty,
\]
and thus integrate to $A$-locally compact operators over any finite interval $(0,x]$.
\end{lemma}
\begin{proof}
Since $\Dom{D_{\pm}^2}=\Dom{S^2}\cap \Dom{T^2}\subset \sF(S,T)$ by
Theorem \ref{p.appl.1}, formula \eqref{eq.KK.4} follows by direct
calculation and \eqref{eq.KK.5} by taking adjoints.  Using
\eqref{eq.KK.4} for $D_+$ and \eqref{eq.KK.5} for $D_-$ it
follows that $D_{\pm}K_{\mu}$ is locally compact. Because of the presence of
the factor $\mu^2$ in this equation, multiplication by $D_{\pm}$ still
yields a family of operators that is uniformly bounded in $\mu$.
\end{proof}

From now on we write $D$ for $D_+$.
Consider the bounded adjointable operators $\chi(D)$ and $\chi(S)$
\begin{align}
  \nonumber       \frac 4{\pi^2} &[\chi(D),\chi(S)]=\\ 
  \label{eq.KK.6}      & \int_0^{1}\int_0^{1} (1+\gl^2S^2)\ii SD(1+\mu^2D^2)\ii 
                                                     +(1+\mu^2D^2)\ii DS (1+\gl^2S^2)\ii d\gl d\mu.
         \end{align}
         
We will show that, for any $\kappa>0$,  a suitable rescaling of the operators $D$ and $S$ gives that 
$[\chi(D),\chi(S)]\geq -\kappa,$ modulo $A$-locally compact operators. We therefore discard the 
multiplicative factor $\frac{4}{\pi^2}$.
We apply the identity 
\[
(1+\gl^2S^2)(1+\gl^2S^2)\ii =(1+\gl^2S^2)\ii (1+\gl^2S^2)=1,
\] 
and multiply the first summand of \eqref{eq.KK.6} from the right and second
summand from the left. The integrand can thus be written as the sum of the
operator
\begin{equation}
\label{eq.KK.7} (1+\gl^2S^2)\ii S D (1+\mu^2D^2)\ii (1+\gl^2S^2)(1+\gl^2S^2)\ii,
\end{equation}
and its adjoint. Expanding $D=S+\go T$ in
\[
  S D (1+\mu^2D^2)\ii (1+\gl^2S^2) = D (1+\mu^2D^2)\ii \gl^2 S^2
           + S (1+\mu^2D^2)\ii D 
\]
gives us a sum of four terms
\begin{align}
\label{eq.KK.8} & S^2(1+\mu^2D^2_+)\ii \gl^2S^2 + S(1+\mu^2D^2)\ii S \\
\label{eq.KK.9} &\quad + S(1+\mu^2D^2)\ii \go T + \gl S\cdot \go T(1+\mu^2D^2_+)\ii S\cdot \gl S.
\end{align}
The summands \eqref{eq.KK.8} are nonnegative and can thus be
discarded.  By adding the adjoints of \eqref{eq.KK.9} and multiply by 
$(1+\gl^2S^2)\ii$ from the left and from the right (\cf \eqref{eq.KK.7}),
we need to address the integral of the sum of operators
\[
      P_\gl\cdot R_\mu\cdot P_\gl + Q_\gl\cdot R_\mu\cdot Q_\gl,
\]
where
\begin{align}
\label{eq.KK.10} R_{\mu}&:= \go T (1+\mu^2D^2)\ii S+S(1+\mu^2D^2)\ii \go T  \\
\label{eq.KK.11} P_\gl &:=(1+\gl^2S^2)\ii ,\quad Q_\gl :=\gl S(1+\gl^2S^2)\ii ,
\end{align}
so that up to positive operators \eqref{eq.KK.6} can be written
\begin{equation}
\label{eq.KK.12}
\int_0^1 P_\gl \left(\int_0^1 R_{\mu} d\mu\right) P_\gl  d\gl 
     +\int_0^1 Q_\gl \left(\int_0^1 R_{\mu} d\mu\right) Q_\gl  d\gl.
\end{equation}
We will prove that  for any $\eps>0$ there is a rescaling of the pair
$(S,T)$ such that the integral $\int_0^1R_{\mu}d\mu \geq -\eps$ modulo
$A$-locally compact operators. By \cite[Lemma 4.3]{MR2016} $P_{\gl}KP_{\gl}$ and $Q_{\gl} KQ_{\gl}$ are $A$-locally compact whenever $K$ is. 
Since $\|P_\gl \|\leq 1$ and $\|Q_\gl \|\leq 1$ 
this allows us to estimate \eqref{eq.KK.12} 
from below as well. Note that since we are integrating over $[0,1]$,
perturbing the integrand by a function $f(\mu)$ with values in the
$A$-locally compact operators that is \emph{uniformly bounded in $\mu$} yields
an $A$-locally compact perturbation after integration. 

Our first goal is to find another expression for $R_{\mu}$. We first consider the algebraic identity 
 \Eqref{eq.A.5} and show that it holds with $a=\sigma_{2} S$ and $b=D=S+\omega T$.

\begin{lemma} The self-adjoint regular operators $S$, $T$ and $D_{\pm}=S\pm\omega T$ satisfy the identities
\begin{align}
\label{eq.A.5.2S}
[(1+D^{2})^{-1},\sigma_{j} S]_{-}&= (1+D^{2})^{-1}[\sigma_j S,D]_{-}D(1+D^{2})^{-1}\\
\nonumber
& \qquad\qquad\qquad+D(1+D^{2})^{-1}[ \sigma_jS, D]_{-}(1+D^{2})^{-1},\\
\label{eq.A.5.2T}
[(1+D^{2})^{-1},\sigma_{j} T]_{-}&=  (1+D^{2})^{-1}[\sigma_j T, D]_{+} D(1+D^{2})^{-1}\\
\nonumber
&\qquad\qquad\qquad-D(1+D^{2})^{-1}[ \sigma_j T, D]_{+}(1+D^{2})^{-1} ,
\end{align}
on $\Dom{S}\cap \Dom{T}$ for $j=1,2$.
\end{lemma}
\begin{proof} 
Recall from the proof of Theorem \ref{p.appl.1} that
the pair $(D, \sigma_j S)$ is weakly commuting for $j=1,2$
and that the pair $(D,\sigma_j T)$ is weakly anticommuting for $j=1,2$. The commutator identities in Lemma \ref{p.A.Leibniz}  in Appendix \ref{s.A} apply including domains
with $b=D_\pm, a=\sigma_j S$ or $a=\sigma_j T$ and $\gl\in i\R\setminus\{0\}$.
We prove \Eqref{eq.A.5.2S} using  the resolvent identities
\[
(1+D^{2})^{-1}=(D+i)^{-1}(D-i)^{-1}=(D-i)^{-1}(D+i)^{-1}.
\]
The Leibniz rules of Lemma \ref{p.A.Leibniz} give the identities
\begin{align*}
[(1+D^{2})^{-1},\sigma_{j} S]_{-}&=(1+D^{2})^{-1}[\sigma_{j} S, D]_{-}(D\pm i)^{-1}+(D\mp i)^{-1}[\sigma_{j} S, D]_{-}(1+D^{2})^{-1},
\end{align*}
\begin{align*}
[(1+D^{2})^{-1},\sigma_{j} T]_{-}&=(1+D^{2})^{-1}[\sigma_{j} T, D]_{+}(D\pm i)^{-1}-(D\mp i)^{-1}[\sigma_{j} T, D]_{+}(1+D^{2})^{-1},
\end{align*}
on $\Dom{S}\cap\Dom{T}$.
Averaging these equalities for $\pm i$ and using that
\[
(D+i)^{-1}+(D-i)^{-1}=2D(1+D^{2})^{-1},
\]
then give us \Eqref{eq.A.5.2S}  and \Eqref{eq.A.5.2T} on $\Dom{S}\cap\Dom{T}$.
\end{proof}
\begin{lemma} \label{p.KK.10}
Recalling the notation $D_\pm=S\pm\go T$ we have for
$\mu>0$ the equality of operators
\begin{multline}\label{eq.KK.13}
   \go R_{\mu} = (1+\mu^2D^2_+)\ii [S,T](1+\mu^2D_-^2)\ii  \\
                    +\mu D_-(1+\mu^2D^2_+)\ii \ovl{[S,T]}(1+\mu^2D_-^2)\ii \mu D_-.
\end{multline}
This amounts to an equality
\begin{multline}
\label{eq.KK.14}\go R_{\mu}=(1+\mu^2D^2_-)\ii [S,T](1+\mu^2D_-^2)\ii \\
       +\mu D_-(1+\mu^2D^2_-)\ii \ovl{[S,T]}(1+\mu^2D_-^2)\ii \mu D_-
\end{multline}
modulo an $A$-locally compact perturbation that is uniformly bounded in $\mu$.
\end{lemma}
\begin{proof}
Note that by definition \Eqref{eq.KK.10} $R_\mu=R(\mu,S,T)$ 
is a rational function of $\mu$ and the (non-commuting) variables 
$S,T$, and we have the relation $R(1,\mu S, \mu T) = \mu^2 R(\mu,S,T)$.
The same is true for the right hand side of \Eqref{eq.KK.13}. Hence
it suffices to prove the claim for $\mu = 1$. It then follows
in general by replacing $S, T, D$ by $\mu S, \mu T, \mu D$ resp.

We have $\go R:=\go R_1 = S(1+D^2)\ii T + T (1+D^2)\ii S$. For the first summand
we calculate on $\domS\cap\domT$ using the commutator identity \Eqref{eq.A.5.2T} and
the Leibniz rule
\begin{align*}
    S(1+D^2)\ii T & = S (1+D^2)\ii \sigma_2 T \sigma_2  \\
      & = \sigma_2 ST (1+D^2)\ii \sigma_2 
            + S [(1+D^2)\ii,\sigma_2 T]_- \sigma_2  \\
            &=\sigma_2 ST (1+D^2)\ii \sigma_2 
            - SD(1+D^{2})^{-1}[ \sigma_2 T, D]_{+}(1+D^{2})^{-1}\sigma_2 \\&\qquad\qquad\qquad + S(1+D^{2})^{-1}[\sigma_2 T, D]_{+} D(1+D^{2})^{-1}\sigma_2&
\end{align*}
and for the second summand using \Eqref{eq.A.5.2S}
\begin{align*}
    T(1+D^2)\ii S & = T (1+D^2)\ii \sigma_2 S \sigma_2 \\
      & = \sigma_2 TS (1+D^2)\ii \sigma_2 
            + T [(1+D^2)\ii,\sigma_2 S]_{-} \sigma_2\\
            &=\sigma_2 TS (1+D^2)\ii \sigma_2 + TD(1+D^{2})^{-1}[ \sigma_j S, D]_{-}(1+D^{2})^{-1}\\ &\qquad\qquad\qquad +T(1+D^{2})^{-1}[\sigma_j S, D]_{-} D(1+D^{2})^{-1}\sigma_{2},
\end{align*}
Adding up and using the identities
\begin{align*}
         D\go = \go D &= \go S + T,\\
         D \sigma_j & = \sigma_j D_-, \quad j=1,2\\
       [\sigma_j S, D ]_- & = - \go \sigma_j [S,T]_+, \quad j=1,2\\
       [\sigma_j T, D ]_+ & = \sigma_j [S,T]_+, \quad j=1,2
\end{align*}

we find
\begin{align*}
    \go R & = \sigma_2 [S,T]_+ (1+D^2)\ii \sigma_2 \\
     &\quad  - D^2(1+D^2)\ii\sigma_2 [S,T]_+ (1+D^2)\ii \sigma_2\\
     & \quad + D_- (1+D^2)\ii \sigma_2[S,T]_+ D(1+D^2)\ii \sigma_2.
\end{align*}
Noting that $1-D^2(1+D^2)\ii = (1+D^2)\ii$ allows to combine the
first and the second summand. Moving the first $\sigma_2$ to the
far right replaces $D$ by $D_-$ in between. Altogether this gives
\[
  \go R= (1+D^2)\ii [S,T]_+ (1+D_-^2)\ii +D_- (1+D^2)\ii [S,T]_+ (1+D_-^2)\ii 
\]
whence the first claim.

By Lemma \ref{eq.KK.4} we can replace $(1+\mu^2D_+^2)\ii $ by
$(1+\mu^2D_-^2)\ii $ in both summands of $R_\mu$ at the cost of an error term that is
uniformly bounded in $\mu$. The latter expression thus equals
\begin{equation*}
(1+\mu^2D^2_-)\ii [S,T](1+\mu^2D_-^2)\ii +\mu D_-(1+\mu^2D^2_-)\ii [S,T](1+\mu^2D_-^2)\ii \mu D_-,
\end{equation*}
modulo $A$-locally compact perturbations that are uniformly bounded in $\mu$.
\end{proof}

We arrive at the following Proposition, which completes the proof of Theorem \ref{thm.unbddprd}.
\begin{prop} \label{localestimate}
Let $E$ be a Hilbert $C^{*}$-module over $B$ and let
$A\to\sL(E)$ be a\linebreak $*$-homomorphism.  Furthermore, let $(S,T)$ be a
weakly anticommuting pair such that $D_{\pm}=S\pm T$ has $A$-locally compact
resolvent.  Then for every $\kappa>0$, $(S,T)$ can be rescaled so that for
$\chi(x):=\frac{2}{\pi}\arctan (x)$ the operators $\chi(S)$ and $\chi(D_+)$
satisfy the operator estimate 
\[
[\chi(S),\chi(D)]\geq -\kappa,
\]
up to an $A$-locally compact perturbation.
\end{prop}
\begin{proof}
Let $\eps>0$ and rescale $(S,T)$ so that the operator estimate \eqref{eq.KK.3} holds true.
We apply Lemma \ref{p.KK.10} to $\hat S, \hat T$. Then the upper left corner
of the corresponding $\hat R_\mu$ gives up to $A$-locally compact
perturbations which are uniformly bounded in $\mu$: 
\begin{align*}
 R_{\mu}&=(1+\mu^2D^2_-)\ii [S,T](1+\mu^2D_-^2)\ii +\mu D_-(1+\mu^2D^2_-)\ii [S,T](1+\mu^2D_-^2)\ii \mu D_-\\
&\geq -\eps(1+ |D_-|)((1+\mu^2D^2_-)^{-2}+(1+\mu^2D^2_-)\ii )\\
&\geq -2\eps (1+ |D_-|)(1+\mu^2D^2_-)\ii \\
&=-2\eps |D_-|(1+\mu^2D^2_-)\ii.
\end{align*}
\Cf Equation \eqref{eq.KK.12} we have, modulo $A$-locally compact perturbations that 
\begin{align*}
\frac{4}{\pi^2}[\chi(D_+),\chi(S)] 
      &=\int_0^{1}P_\gl \left(\int_0^{1}R_{\mu}d\mu\right)P_\gl d\gl 
                   + \int_0^{1}Q_\gl \left(\int_0^{1}R_{\mu}d\mu\right)Q_\gl d\gl\\
      &\geq -4\eps\int_0^{1} |D_-|(1+\mu^2D^2_-)\ii d\mu\\
      &\geq -2\pi\eps,
\end{align*}
since
\[
   \pm \arctan (|D_-|)=\pm \int_0^{1}|D_-|(1+\mu^2D_-^2)\ii d\mu\leq \frac{\pi}{2}.
\]
Thus, choosing $\eps=\frac{2\kappa}{\pi^{3}}$ and rescaling $(S,T)$
according to Proposition \ref{estimates} yields that $[\chi(D_+),\chi(S)]\geq
-\kappa$ modulo $A$-locally compact perturbations.
\end{proof}


\appendix
\section{Commutator identities}
\label{s.A}
 
We collect here some useful identities for (graded) commutators. In the sequel
$a,b,c,\ldots$ denote elements in a unital $\C$--algebra. This section is concerned
only with algebraic identities. When applying to unbounded operators the
equality of domains needs to be checked separately.

Recall from Sections \ref{ss.WAO}, \ref{ss.CPNWA} 
\begin{equation}\label{eq.A.1}
    [a,b]_\tau := a\cdot b + \tau b\cdot a, \quad \tau\in \{+, -\}.
\end{equation}

\begin{lemma}\label{p.A.Leibniz}
For $\sigma,\tau\in\{+,-\}$ one has the Leibniz rules
\begin{align}
 \label{eq.A.2}   [a,b\cdot c]_\tau & = [a,b]_\sigma\cdot c - \sigma b\cdot [a,c]_{-\sigma\tau},\\
 \label{eq.A.3}  [a\cdot b, c]_\tau & = a\cdot [b,c]_\sigma - \sigma [a,c]_{-\sigma\tau} \cdot b.
\end{align}
\end{lemma}
This follows immediately by expanding the left and right hand sides.

\begin{lemma}\label{p.A.ResolventCommutator}
Assume that for $\gl\in\C$ the element $\pm\gl+b$ resp. $\gl+b^2$ is invertible;
$\gl$ is an abbreviation for $\gl\cdot\one$. Then
\begin{align}
 \label{eq.A.4}   (\gl + b )\ii a  &= a (\gl - \tau b)\ii - (\gl+b)\ii [b,a]_\tau (\gl - \tau b)\ii,\\
 \label{eq.A.5}   [ (\gl + b^2 )\ii, a]_-  &=  (\gl + b^2)\ii b [a,b]_- (\gl+b^2)\ii\nonumber \\
                            & \qquad +(\gl + b^2)\ii  [a,b]_- b(\gl+b^2)\ii,\\
 \label{eq.A.6}   [ (\gl + b^2 )\ii, a]_-  &=  - (\gl + b^2)\ii b [a,b]_+ (\gl+b^2)\ii\nonumber \\
                            & \qquad + (\gl + b^2)\ii  [a,b]_+ b(\gl+b^2)\ii.
\end{align}
\end{lemma}
\begin{proof}We have
\[
    (\gl+b) a = ba + \tau ab - \tau ab + a\gl = [b,a]_\tau + a (\gl - \tau b).
\]
Now multiply from the left by $(\gl+b)\ii$ and from the right by 
$(\gl-\tau b)\ii$ to obtain the first identity.

The second and third identity follow by applying the first identity to
$b^2, a$ and then the Leibniz rule to $[b^2,a]_- = b[b,a]_-+[b,a]_-b 
= b[b,a]_+ - [b,a]_+ b$.
\end{proof}

\bibliography{mlbib,localbib}
\bibliographystyle{amsalpha-lmp}

\end{document}